\documentclass[a4paper, 10pt]{amsart}
\usepackage{latexsym}
\usepackage{tikz-cd}
\usepackage{amssymb}
\usepackage{xcolor}
\usepackage{colortbl}
\usepackage[normalem]{ulem}

\newtheorem*{KN}{Crazy Knight's Tour Problem}

\def\G{\Gamma}

\def\Z{\mathbb{Z}}
\def\N{\mathrm{N}}

\newcommand{\probname}{Crazy Knight's Tour Problem}

\def\lcm{\mathrm{lcm}}

\def\E{\mathcal{E}}

\def\R{\mathcal{R}}
\def\C{\mathcal{C}}

\usepackage{mathtools}

\usepackage{geometry}
\usepackage[utf8]{inputenc} 
\usepackage[T1]{fontenc}
\newtheorem{thm}{Theorem}[section]

\newtheorem{cor}[thm]{Corollary}
\newtheorem{prop}[thm]{Proposition}

\newtheorem{rem}[thm]{Remark}

\theoremstyle{definition}
\newtheorem{defi}[thm]{Definition}
\newtheorem{ex}[thm]{Example}
\numberwithin{equation}{section}

\def\Z{\mathbb{Z}}

\title[$\lambda$-Fold Non-zero sum Heffter arrays]{Existence of $\lambda$-Fold Non-zero sum Heffter arrays through local considerations}
\author{Simone Costa}
\address{DICATAM, Universit\`a degli Studi di Brescia, Via
Branze 43, 25123 Brescia, Italy}
\email{simone.costa@unibs.it}
\author{Stefano Della Fiore}
\address{DII, Universit\`a degli Studi di Brescia, Via
Branze 38, 25123 Brescia, Italy}
\email{s.dellafiore001@unibs.it}
\begin{document}
\begin{abstract}
In \cite{CostaDellaFiorePasotti} was introduced, for cyclic groups, the class of partially filled arrays of the \emph{non-zero sum Heffter array} that are, as the Heffter arrays, related to
difference families, graph decompositions, and biembeddings. Here we generalize this definition to any finite groups.

Given a subgroup $J$ of order $t$ of a group $G$, a \emph{$\lambda$-fold non-zero sum Heffter array} over $G$ relative to $J$, $^\lambda\N\mathrm{H}_t(m,n; h,k)$,  is an $m \times n$ p. f. array with entries in $G$ such that:
each row contains $h$ filled cells and each column contains $k$ filled cells;
for every $x\in G\setminus J$, the sum of the occurrence of $x$ and $-x$ is $\lambda$;
the sum of the elements in every row and column is, following the natural orderings from left to right for the rows and from top to bottom for the columns, different from $0$ (in $G$).

In \cite{CostaDellaFiorePasotti}, there was presented a complete, probabilistic, solution for the existence problem in case $\lambda=1$ and $G=\Z_v$ that is the starting point of this investigation. In this paper, we will consider the existence problem for a generic value of $\lambda$ and a generic finite group $G$, and we present an almost complete solution to this problem. In particular, we will prove, through local considerations (inspired by Lov\'asz Local Lemma), that there exists a \emph{$\lambda$-fold non-zero sum Heffter array} over $G$ relative to $J$ whenever the trivial necessary conditions are satisfied and $|G|=v\geq 41$. This value can be turned down to $29$ in case the array does not contain empty cells.

Finally, we will show that these arrays give rise to biembeddings of multigraphs into orientable surfaces and we provide new infinite families of such embeddings.
\end{abstract}
\maketitle
\section{Introduction}
An $m\times n$ partially filled (p.f., for short) array on a set $\Omega$ is an $m \times n$ matrix whose elements belong to $\Omega$
and where some cells can be empty. In 2015 Archdeacon \cite {A} introduced a class of p.f. arrays which has been extensively studied:
the \emph{Heffter arrays}.
\begin{defi}\label{def:H}
A \emph{Heffter array} $\mathrm{H}(m,n; h,k)$ is an $m \times n$ p. f. array with entries in $\Z_{2nk+1}$ such that:
\begin{itemize}
\item[(\rm{a})] each row contains $h$ filled cells and each column contains $k$ filled cells;
\item[(\rm{b})] for every $x\in \Z_{2nk+1}\setminus\{0\}$, either $x$ or $-x$ appears in the array;
\item[(\rm{c})] the elements in every row and column sum to $0$ (in $\Z_{2nk+1}$).
\end{itemize}
\end{defi}
Trivial necessary conditions for the existence of an $\mathrm{H}(m,n; h,k)$ are $mh=nk$, $3\leq h\leq n$ and $3\leq k
\leq m$. If $m=n$ then $h=k$ and a square $\mathrm{H}(n,n;k,k)$ will be simply denoted by $\mathrm{H}(n;k)$.
For those kinds of arrays, the existence problem has been completely solved. Indeed, in \cite{ADDY, CDDY, DW}, it was proven the following result.
\begin{thm}\label{thm:Heffter}
An $\mathrm{H}(n;k)$ exists for every $n\geq k\geq 3$.
\end{thm}

In \cite{CPEJC} this concept has been generalized as follows.

\begin{defi}\label{def:lambdaRelative}
Let $v=\frac{2nk}{\lambda}+t$ be a positive integer,
where $t$ divides $\frac{2nk}{\lambda}$, and
let $J$ be the subgroup of $\Z_{v}$ of order $t$.
A $\lambda$-\emph{fold Heffter array $A$ over $\Z_{v}$ relative to $J$}, denoted by $^\lambda\mathrm{H}_t(m,n; h,k)$, is an $m\times n$ p.f. array
with elements in $\Z_{v}$ such that:
\begin{itemize}
\item[($\rm{a_1})$] each row contains $h$ filled cells and each column contains $k$ filled cells;
\item[($\rm{b_1})$] the multiset $\{\pm x \mid x \in A\}$ contains each element of $\Z_v\setminus J$ exactly $\lambda$ times;
\item[($\rm{c_1})$] the elements in every row and column sum to $0$ (in $\Z_v$).
\end{itemize}
\end{defi}
It is easy to see that if $\lambda=t=1$ we find again the arrays of Definition \ref{def:H}.

Classical Heffter arrays and the above generalization have been introduced also because of their vast variety of applications and
connections with other much-studied problems and concepts. In particular, there are some recent papers in which Heffter arrays
are investigated to obtain new face $2$-colorable embeddings (briefly biembeddings) see \cite{A, CDY, CMPPHeffter, CPPBiembeddings, CPEJC, CostaPasotti2, DM}.
On the other hand, several authors focused their attention on the existence problem of Heffter arrays and their variations, see \cite{ABD, ADDY, BCDY, CDDY, RelH,DW,PM2, MP,MP2,MP3}. For a complete discussion on this topic we refer the reader to the survey \cite{DP} and the references therein.

In \cite{CostaDellaFiorePasotti}, it was presented a new class of p.f. arrays, which is related to that of Heffter arrays: the non-zero sum Heffter arrays. This definition was motivated by Alspach's partial sums conjecture that deals with sets whose sums are different from $0$, see \cite{AL,ADMS, BH, HOS, O}. Here we propose a generalization of this concept to a generic finite group (as done in \cite{CPPBiembeddings} for the classical Heffter arrays).
\begin{defi}\label{def:NZS}
Let $G$ be a finite group and
let $J$ be a subgroup of $G$ of order $t$.
A $\lambda$-\emph{fold non-zero sum Heffter array $A$ over $G$ relative to $J$}, denoted by $^\lambda \N\mathrm{H}_t(m,n; h,k)$, is an $m\times n$ p.f. array
with elements in $G$ such that:
\begin{itemize}
\item[($\rm{a_2})$] each row contains $h$ filled cells and each column contains $k$ filled cells;
\item[($\rm{b_2})$] the multiset $\{\pm x \mid x \in A\}$ contains $\lambda$ times each element of $G\setminus J$;
\item[($\rm{c_2})$] the sum of the elements in every row, following the natural ordering from left to right, and column, following the natural ordering from top to bottom, is different from $0$ (in $G$).
\end{itemize}
Here if $\lambda=1$,  $t = 1$,  $m=k$ or $m=n$, we will use, respectively, the notations $\N\mathrm{H}_t(m,n;h,k)$,  $^\lambda\N\mathrm{H}(m,n;h,k)$, $^\lambda\N\mathrm{H}_t(m,n)$ and $^\lambda\N\mathrm{H}_t(n;k)$.
\end{defi}
In \cite{CostaDellaFiorePasotti}, the authors focused on the case $\lambda=1$, namely when the entries of the array are pairwise distinct and consider only cyclic groups (see also \cite{Costa} and \cite{PM} for other results on the cyclic case). In particular, they prove the existence of an $\N\mathrm{H}_t(m,n; h,k)$ over $\Z_v$ relative to $J$ whenever the trivial necessary conditions $nk=mh$, $v=\frac{2nk}{\lambda}+t$ and $t|v$ are satisfied.
In this paper we will consider the existence problem for a generic value of $\lambda$ and we present an almost complete solution to this problem.
For this purpose,  in Section \ref{necessarySection},  we will first determine more general necessary conditions for the existence problem from which it follows a complete solution for $v = 2$.
Then, in Section $3$, we will try to solve this problem with direct constructions. In particular, we will present a complete solution also when $v=3$ and in the case of square arrays over abelian groups of odd order. In Section $4$, we will revisit the construction of \cite{CostaDellaFiorePasotti} proving that there exists a $^\lambda\N\mathrm{H}_t(m,n; h,k)$ over $G$ relative to $J$ whenever the necessary conditions are satisfied provided that $h$ and $k$ are big enough if $\lambda$ is considered fixed. With this procedure, for any group $G$, we leave the problem open for an infinite number of arrays. This issue is due to the fact that, generating uniformly at random the matrix, the probability of the event that a given row (or a given column) sums to zero is not, in general, independent from all the other entries of the array. Then, inspired by the Lov\'asz Local Lemma (see \cite{LLL}), we impose a different distribution: we tessellate the array with suitable tiles and we choose the elements of a tile uniformly at random among a small prescribed set. Within these local considerations, we reduce the number of dependencies among the events mentioned above.
It will follow that there exists a \emph{$\lambda$-fold non-zero sum Heffter array} over $G$ relative to $J$ whenever the necessary conditions are satisfied and $|G|=v\geq 41$ leaving the problem open only for a finite set of groups. This value can be turned down to $29$ in case the array does not contain empty cells. Finally, in the last section of this paper, we show that these arrays give rise to biembeddings of multigraphs into orientable surfaces and we provide infinite families of such embeddings.
\section{Necessary conditions}\label{necessarySection}
In \cite{CostaDellaFiorePasotti}, it was proved that an $\N\mathrm{H}_t(m,n; h,k)$ over $\Z_v$ relative to $J$ exists whenever the trivial necessary conditions $nk=mh$, $v=\frac{2nk}{\lambda}+t$ and $t|v$ are satisfied. Those conditions can be easily generalized to the following ones.
\begin{rem}\label{necessary}
We note that, if $G$ is a group of order $v$, $J$ is a subgroup of $G$ of order $t$, then a $^\lambda\N\mathrm{H}_t(m,n; h,k)$ over $G$ relative to $J$ can exists only if the following necessary conditions hold:
\begin{itemize}
\item[(a)] $nk=mh$;
\item[(b)] $v=\frac{2nk}{\lambda}+t$;
\item[(c)] $t|v$;
\item[(d)] if $G\setminus J$ contains an involution, then $\lambda$ must be even.
\end{itemize}
Note that the latter condition is automatically satisfied in the case of cyclic groups.
\end{rem}
Now we will see that in the case of groups $\mathbb{Z}_2^r$ those conditions are not sufficient. Indeed, for such groups, we can state more strict conditions.
We begin by presenting the case $G=\mathbb{Z}_2$.
\begin{thm}\label{th:ExistenceLambdaV2}
For $m \geq k$ and $n \geq h$, there exists a $^\lambda\N\mathrm{H}(m, n; h, k)$ over $\mathbb{Z}_2$ whenever the necessary conditions of Remark \ref{necessary} are satisfied and $h\cdot k$ is odd.
\end{thm}
\begin{proof}
Thanks to Lemma 4.1 of \cite{CostaDellaFiorePasotti}, we know that there exists an $m \times n$ partially filled array $A$ that has exactly $h$ filled cells in each row and $k$ filled cells in each column. Then the theorem trivially follows since we can fill the array $A$ only with the element $1$ of $\mathbb{Z}_2$ that is an involution.
\end{proof}

To present our result on groups $\mathbb{Z}_2^r$ it will be useful to introduce some notation.
The rows and the columns of an $m\times n$ array $A$ are denoted by $R_1,\ldots, R_m$ and by $C_1,\ldots, C_n$, respectively. Also, we name by $\E(A)$, $\E(R_i)$, $\E(C_j)$ the list of the elements of the filled cells of $A$, of the $i$-th row and of the $j$-th column, respectively. Then we denote by $skel(A)$ the \emph{skeleton} of $A$, that is the set of the filled positions of $A$. Also, given a multiset $M$, we define $M^{\lambda}$ to be the multiset in which each element of $M$ (counted with its multiplicity) is repeated $\lambda$ times. 
\begin{thm}\label{z2h}
Let $G$ be isomorphic to $\mathbb{Z}_2^r$ and let $J$ be a proper subgroup of $G$ that is isomorphic to $\mathbb{Z}_2^{r_1}$. Then there exists a $^\lambda\N\mathrm{H}_t(1,n)$ over $G$ relative to $J$ if and only if the necessary conditions of Remark \ref{necessary} are satisfied and one of the following conditions holds:
\begin{itemize}
\item[1)] $t=2$ (i.e. $r_1=1$) and $n$ is an odd multiple of $v-t$;
\item[2)] $v=2$ (i.e. $r=1$ and $r_1=0$)  and $n$ is odd.
\end{itemize}
\end{thm}
\proof
It is easy to check that the sum of the elements of $\mathbb{Z}_2^r$ is non-zero only if $r=1$. It follows that the sum of the elements of $G\setminus J$ is non-zero only when $t$ or $v$ are equal to $2$.

Let us assume that there exists $^\lambda\N\mathrm{H}_t(1,n)$ over $G$ relative to $J$ and let us denote it by $A$.
Since $G\setminus J$ only contains involutions we have that:
$$\E(A)=\E(R_1)=(G\setminus J)^{\lambda/2}.$$
Clearly, $A$ can exist only if the sum of the elements of $G\setminus J$ is non-zero. So we can assume that $t=2$ or $v=2$.

CASE 1: $t=2$. In this case, the sum of the elements of $G\setminus J$ is non-zero but it is an involution of $G$ ($G\setminus \{0\}$ only contains involution). It follows that
$$\sum_{a\in \E(R_1)}a=\lambda/2\sum_{g\in G\setminus J}g\not=0 $$
only if $\lambda/2$ is odd, or equivalently, if $n$ is an odd multiple of $v-t$.

The thesis follows since the case $v=2$ has already been considered in Theorem \ref{th:ExistenceLambdaV2}.
\endproof

\begin{rem}
It must be noted that, for $v=2$ and for the arrays of just one row/column,  the necessary conditions of Remark \ref{necessary} are not sufficient.  Theorems \ref{th:ExistenceLambdaV2} and \ref{z2h} are the only theorems presented in this paper where we need more restricted conditions.  Hence we believe that the smallest cases of $v$ (and the arrays with just one row/column) are, somehow, intrinsically different from the general case.
\end{rem}

\section{Some direct constructions}
The goal of this section would have been to solve this problem with direct constructions. Even though we fail to achieve this goal we are able to present a complete direct solution when $v=3$ and in the case of square arrays over abelian groups of odd order.
\subsection{A complete solution for $v=3$}
\begin{thm}\label{th:ExistenceLambdaV3}
There exists a $^\lambda\N\mathrm{H}(m, n; h, k)$ over $\mathbb{Z}_3$ whenever the necessary conditions of Section \ref{necessarySection} are satisfied.
\end{thm}
\begin{proof}
Since $nk=mh$ we have that $\frac{nk}{\lcm(m,n)}$ is an integer and we denote it by $r:=\frac{nk}{\lcm(m,n)}$.
Then we have that also the following quantities are integers
$$\frac{k}{r}=\frac{\lcm(m,n)}{n} \quad \mbox{ and } \quad \frac{h}{r}=\frac{\lcm(m,n)}{m}.$$
Hence we can define the following set of cells
$$Q:=\{(i,j):\ 1\leq i\leq k/r; \ 1\leq j\leq h/r\}.$$
Then, since
$$\frac{r\cdot n}{h}=\frac{r\cdot m}{k}=\frac{mn}{\lcm(m,n)}=\gcd(m,n),$$
we can consider the subset of the $m\times n$ array defined by:
$$B:= \bigcup_{j=0}^{r-1} \left(\bigcup_{i=0}^{r n/h-1} Q+ j (0,h/r)+i(k/r,h/r)\right)$$
where all the arithmetic on the row and column indices is performed modulo $m$ and $n$ respectively and the sets of reduced residues are $\{1,2,\ldots, m\}$ and $\{1,2,\ldots,n\}$.

Here,  we fill all the cells of $B$ with $1$'s obtaining the array $A_1$. Now, if $3 \mid k$ then we replace all the elements in the set of cells $H_1$ by $-1$'s where the indexes are considered modulo $m$ and where
$$
H_1 := \bigcup_{i=1}^{\lceil n/h\rceil} skel(R_{ik}).
$$
We denote by $A_2$ the array so defined. Here, if $3$ does not divide $k$ we set $A_2=A_1$.
Since every column intersects the rows of $H_1$ in either $1$ or $2$ cells, the sum of the elements in any given column of $A_2$ is non-zero.

Similarly, if $3 \mid h$ then we change the sign to all elements of $A_2$ in the set of cells $H_2$, i.e. the $1$'s are replaced by $-1$'s and vice versa, where
$$
H_2 := \bigcup_{j=1}^{\lceil m/k\rceil} skel(C_{jh}).
$$
We denote by $A$ the array so defined. Here, if $3$ does not divide $h$ we set $A=A_2$.
Since every row intersects the columns of $H_2$ in either $1$ or $2$ cells, the sum of the elements in any given row of $A$ is non-zero. Moreover, since every column of $A$ intersects $H_2$ in either $0$ or $k$ cells and the sum of the elements in any given column of $A_2$ is non-zero, also the sum of all the elements in any column of $A$ is non-zero. Hence, we obtain a $^\lambda \N\mathrm{H}(m, n;h, k)$ over $\mathbb{Z}_3$ such that $skel(A)=B$.
\end{proof}
\begin{ex}
In the following we provide an example of the sets $B$, $H_1$ and $H_2$ for $m=n=7$ and $k=h=3$. The black dots represent the cells contained in the set $B$, the green cells are the ones contained in $H_1 \cap H_2$, the red cells are the ones contained in $H_1 \setminus H_2$ and finally, the yellow cells are the ones in $H_2 \setminus H_1$.
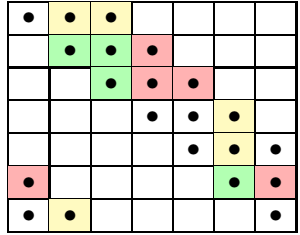
\begin{center}
$\begin{array}{|r|r|r|r|r|r|r|}
\hline \bullet & \cellcolor{yellow!30}\bullet & \cellcolor{yellow!30} \bullet & \:\: & \:\: & \:\: & \:\: \\
\hline & \cellcolor{green!30}\bullet & \cellcolor{green!30} \bullet & \cellcolor{red!30}\bullet & \:\: & \:\: & \:\: \\
\hline & & \cellcolor{green!30} \bullet & \cellcolor{red!30}\bullet & \cellcolor{red!30}\bullet & \:\: & \:\: \\
\hline \:\: & \:\: & \:\: & \bullet & \bullet & \cellcolor{yellow!30} \bullet & \:\:\\
\hline & & & \:\: & \bullet & \cellcolor{yellow!30} \bullet & \bullet \\
\hline \cellcolor{red!30}\bullet & & & \:\: & \:\: & \cellcolor{green!30} \bullet & \cellcolor{red!30}\bullet \\
\hline \bullet & \cellcolor{yellow!30}\bullet & \:\: & \:\: & \:\: & \:\: & \bullet \\
\hline
\end{array}$
\end{center}
Then, applying the proof of Theorem \ref{th:ExistenceLambdaV3} we get the following $^{21} \N\mathrm{H}(7; 3)$ over $\mathbb{Z}_3$.
\begin{center}
$\begin{array}{|r|r|r|r|r|r|r|}
\hline 1 & -1 & -1 & \:\: & \:\: & \:\: & \:\: \\
\hline & 1 & 1 & -1 & \:\: & \:\: & \:\: \\
\hline & & 1 & -1 & -1 & \:\: & \:\: \\
\hline \:\: & \:\: & \:\: & 1 & 1 & -1 & \:\:\\
\hline & & & \:\: & 1& -1 & 1 \\
\hline -1 & & & \:\: & \:\: & 1& -1 \\
\hline 1 & -1 & \:\: & \:\: & \:\: & \:\: & 1 \\
\hline
\end{array}$
\end{center}
\end{ex}
\begin{ex}
Applying the proof of Theorem \ref{th:ExistenceLambdaV3} we get the following $^{12} \N\mathrm{H}(4, 6; 3, 2)$ over $\mathbb{Z}_3$.
\begin{center}
$\begin{array}{|r|r|r|r|r|r|r|}
\hline 1 & 1 & -1 & \:\: & \:\: & \:\: \\
\hline 1 & 1 & -1 & \:\: & \:\: & \:\:\\
\hline \:\: & \:\: & \:\: & 1 & 1 & -1\\
\hline \:\: & \:\: & \:\: & 1 & 1 & -1\\
\hline
\end{array}$
\end{center}
\end{ex}

\subsection{On abelian groups $G$ of odd order}

\begin{thm}\label{th:ExistenceLambdaAbelianNoInv}
Let $G$ be a finite abelian group of odd order $v$ and let $J$ be a subgroup of $G$ of order $t$. For $n\geq k$, there exists a $^\lambda\N\mathrm{H}_t(n; k)$ over $G$ relative to $J$ whenever the necessary conditions of Section \ref{necessarySection} are satisfied.
\end{thm}
\begin{proof}
First of all, we consider the case where $v-t\geq 4$.

Let $Q$ be the set of cells given by $Q:=\{(i,1):\ 1\leq i\leq k\}$.
We consider the subset of the $n \times n$ array defined by:
$$B:=\bigcup_{i=0}^{n-1} \left( Q+i(1,1)\right)$$
where the sum is considered modulo $n$ and the set of reduced residues is $\{1,2,\ldots,n\}$.

Let us order the elements of $G\setminus J$ as $g_1,\dots,g_{v-t}$ with the additional property that if $x\in \{g_1,\dots,g_{\frac{v-t}{2}}\}$ then $-x\in \{g_{\frac{v-t}{2}+1},\dots,g_{v-t}\}$. This is possible since $v$ is odd and hence also $t$ must be odd. Then we consider the set $\Omega=\{g_1,\dots,g_{\frac{v-t}{2}}\}^{\lambda}$ and an array $A_1$ such that $skel(A_1)=B$ and $\E(A_1)=\Omega$.

Clearly, if $k=1$, however we fill $B$ with the elements of $\Omega$, we obtain a $^\lambda\N\mathrm{H}_t(n; k)$. So we may assume that $k\geq 2$.
In this case, for each $i \in [1,n]$,  we denote by $y_i$ and $x_i$ the elements of the $i$-th column $C_i$ that are in positions $(i, i)$, $(i+1, i)$ respectively (where the sum is considered modulo $n$). Since $\frac{v-t}{2}\geq 2$, we can choose $y_i$ in such a way that ${y_i\pm x_i\not=0}$. It suffices to set $x_i=g_{2i+1}$ and $y_i=g_{2i}$ where the indexes are considered modulo $\frac{v-t}{2}$ and then to choose,  arbitrarily among the remaining elements of $\Omega$, the rest of the array.

We claim that there exists a choice between $x_i$ and $-x_i$ such that the sum of the elements in the $i$-th column is different from zero either using $y_i$ or $-y_i$.

Indeed, if $x_i$ does not have this property, we have that either
\begin{equation}\label{sommazero1}	\left(\sum_{a \in \E(C_i) \setminus \{x_i, y_i \}} a \right)+ x_i + y_i = 0
\end{equation}
or
\begin{equation}\label{sommazero2}	\left(\sum_{a \in \E(C_i) \setminus \{x_i, y_i \}} a \right)+ x_i - y_i = 0.
\end{equation}

If equation \eqref{sommazero1} holds then we have that
$$
\left(\sum_{a \in \E(C_i) \setminus \{x_i, y_i \}} a \right)- x_i + y_i \not = 0.$$
Moreover, since the group $G$ is involution-free (its order is odd) and since $y_i \neq - x_i$ by construction, we have that 
$$
	x_i + y_i \neq - \left(x_i + y_i\right) = -x_i - y_i,
$$
where in the last equality we used the abelianility of $G$.  Hence we also have that
$$ \left(\sum_{a \in \E(C_i) \setminus \{x_i, y_i \}} a \right)- x_i - y_i \not = 0.
$$

If equation \eqref{sommazero2} holds, instead, we have that
$$
\left(\sum_{a \in \E(C_i) \setminus \{x_i, y_i \}} a \right)- x_i - y_i \not= 0.$$
Moreover,  as done before,  since the group $G$ is abelian and involution-free (its order is odd) and since $y_i \neq x_i$ by construction, we also have that
$$\left(\sum_{a \in \E(C_i) \setminus \{x_i, y_i \}} a \right)- x_i + y_i \not= 0.
$$

Based on this simple argument, we can assume that the element $x_i$ in position $(i+1,i)$ is such that the sum of the elements in the $i$-th column $C_i$ is different from zero either using $y_i$ or $-y_i$. We denote by $A_2$ the array so defined.
Then, we replace the element $y_i$ in position $(i,i)$ of $A_2$ by $-y_i$ for each row $R_i \in [1,n]$ that sums to zero. It follows that we obtain a $^\lambda \N\mathrm{H}_t(n;k)$ over $G$ relative to $J$ whenever $v-t\geq 4$.

Now we observe that, if $v\geq 5$ is odd, we have that $v-t\geq 4$ for any proper divisor $t$ of $v$. The thesis follows since the case $v=3$ has been already considered in Theorem \ref{th:ExistenceLambdaV3}.
\end{proof}

\section{Our Method}
The direct constructions presented in the previous section solve only some particular instances of the existence problem. On the other hand, we do not believe it would be possible to adapt these ideas to the general case since in both Theorem \ref{th:ExistenceLambdaV3} and Theorem \ref{th:ExistenceLambdaAbelianNoInv} we use the fact that the group $G$ is involution free. Therefore, in order to attack this existence problem, we develop now a local variant (inspired by Lov\'asz Local Lemma) of the probabilistic method used in \cite{CostaDellaFiorePasotti}.
\subsection{Bound with Expected values}
First of all, we revisit the method of \cite{CostaDellaFiorePasotti} in order to obtain some sufficient conditions on the existence of a $\lambda$-fold non-zero sum Heffter array.

\begin{thm}\label{th:ExistenceLambdaCyclic}
Let $G$ be a group of size $v$ and let $J$ be a subgroup of $G$ of size $t$. Then there exists a $^\lambda \N\mathrm{H}_t(m,n;h,k)$ over $G$ relative to $J$ if the necessary conditions of Section \ref{necessarySection} are satisfied and the following inequality holds
$$
\lambda \left(\frac{m}{mh - h + 1} + \frac{n}{nk - k + 1}\right) < 1\,.
$$
\end{thm}
\begin{proof}

Let $\Omega$ be a multiset of elements of $G$ whose size is $nk$ and such that $\pm \Omega = \left(G \setminus J\right)^{\lambda}$.
We denote by $B$ the set of cells of an $m\times n$ partially filled array $A$ having exactly $h$ filled cells in each row and exactly $k$ filled cells in each column that exists due to Lemma 4.1 of \cite{CostaDellaFiorePasotti}. Note that here we have that $skel(A)=B$.

We prove that we can fill the cells of $B$ with the elements of $\Omega$ in such a way that the sum over each row and each column is non-zero.
Let us pick uniformly at random the elements of $A$ among the set $\Omega$. In this way, we will ensure that $\E(A)=\Omega$.
We denote by $X_i$ the event that the $i$-th row sums to zero and by $\mathbb{P}(X_i)$ its probability. Also we name by $X$ the random variable given by the number of rows that sum to zero and by $\mathbb{E}(X)$ its expected value.

Due to the linearity of the expected value, we have that:
$$\mathbb{E}(X)=\sum_{i=1}^m \mathbb{P}\left(X_i\right).$$
By symmetry, $\mathbb{P}(X_i)=\mathbb{P}(X_1)$ for every $i=1,\ldots,m$, and hence
$$\mathbb{E}(X)=m \cdot \mathbb{P}\left(X_1\right).$$

If we have already chosen the first, following the natural ordering from left to right, $h-1$ elements, $x_1,\dots,x_{h-1}$ of the first row then there exists at most $\lambda$ elements $\bar{x}\in \Omega\setminus\{x_1,\dots,x_{h-1}\}$ that make the sum zero. This means that
$$\mathbb{P}\left(X_1\right)\leq \frac{\lambda}{mh-(h-1)} $$
and thus
$$\mathbb{E}(X)\leq \frac{\lambda m}{mh-(h-1)}.$$

Analogously, if we denote by $\mathbb{E}(Y)$ the expected value of the random variable $Y$ given by the number of columns that sum to zero then we obtain
$$\mathbb{E}(Y)\leq \frac{\lambda n}{nk-(k-1)}.$$

Therefore we can fill $B$ with the elements of $\Omega$ such that the sum over each row and column is non-zero whenever $\mathbb{E}(X) + \mathbb{E}(Y) < 1$.
\end{proof}

\begin{rem}
Note that the inequality in Theorem \ref{th:ExistenceLambdaCyclic} is satisfied when the parameters $h$ and $k$ are sufficiently large if $\lambda$ is considered fixed.
\end{rem}

In the next sections, we improve the result of Theorem \ref{th:ExistenceLambdaCyclic} by proving that there exists a $\lambda$-fold non-zero sum Heffter array over $G \setminus J$ for $v=|G|$ larger than a small constant that does not depend on $\lambda$.

\subsection{Tiling in LLL style}
The procedure of the previous paragraph leaves, for any group $G$, the problem open for an infinite number of arrays. This problem is caused by the fact that, generating uniformly at random the matrix, the probability of the event $X_i$ (or $Y_i$) that a given row (or a given column) sums to zero depends on all the entries of the array.
Here, inspired by the Lov\'asz Local Lemma (see \cite{LLL}), we impose a different distribution: we tessellate the array with suitable tiles and we choose the elements of a tile uniformly at random among a small prescribed set. With these local considerations, we reduce the number of dependencies among the events.

First of all, we define the set of admissible tiles we will use as bricks in our construction.
\begin{defi}\label{nice}
A set $T$ of cells in a $m\times n$ array is said to be a \emph{nice tile} if, for any group $G$, for any subset $S$ of $G$ of size $|T|$, and for any vectors $(r_1,\dots,r_m)$ and $(c_1,\dots,c_n)$ of, respectively, $G^m$ and $G^n$, there exists a partially filled array $A$ such that:
\begin{itemize}
\item[$(a)$] $skel(A)=T$;
\item[$(b)$] $\E(A)=S$;
\item[$(c_1)$] if the $i$-th row of $A$ is non-empty, there exist $\beta_i \in [1,n]$ such that
$$skel(R_i)=\{(i,\beta_i),(i,\beta_i+1),\dots ,(i,\beta_i+h_i-1)\}$$
where $h_i=|T\cap skel(R_i)|$ and the indexes are taken modulo $n$;
\item[$(c_2)$] if the $j$-th column of $A$ is non-empty, there exist $\gamma_j\in [1,m]$ such that
$$skel(C_j)=\{(\gamma_j,j),(\gamma_j+1,j),\dots ,(\gamma_k+k_j-1,j)\}$$
where $k_j=|T\cap skel(C_j)|$ and the indexes are taken modulo $m$;
\item[$(d_1)$] if the $i$-th row of $A$ is non-empty, the sum of its elements, following the natural ordering from left to right and starting from the cell $(i,\beta_i)$, is different from $r_i$;
\item[$(d_2)$] if the $j$-th column of $A$ is non-empty, the sum of its elements, following the natural ordering from top to bottom and starting from the cell $(\gamma_j,j)$, is different from $c_j$.
\end{itemize}
\end{defi}

\begin{thm}\label{tiling}
Let $B$ be a set of cells in a $m\times n$ array that has exactly $h$ filled cells in each row and $k$ in each column. Then, given a group $G$ of size $v$ and a subgroup $J$ of $G$ of size $t$, there exists a $^\lambda \N\mathrm{H}_t(m,n;h,k)$ over $G$ relative to $J$ whenever the necessary conditions of Section \ref{necessarySection} are satisfied and:
\begin{itemize}
\item[1)] it is possible to partition $B$ into nice tiles $T_1,\dots,T_{\ell}$;
\item[2)]$v-t\geq \max_i(|T_i|).$
\end{itemize}
\end{thm}
\proof
We first assume that $\lambda$ is even. In this case we note that $\Omega=(G\setminus J)^{\lambda/2}$ is such that $\pm \Omega=(G\setminus J)^{\lambda}$.
Moreover, since the necessary conditions are satisfied, and since the tiles partition $B$, we have that
$$\frac{\lambda}{2}|G\setminus J|=\frac{\lambda}{2}(v-t)=nk=\sum_{i=1}^{\ell} |T_i|.$$
Now we enumerate the $v-t$ elements of $G\setminus J$ as $g_1,\dots,g_{v-t}$.
Here we want to define, recursively, the sets $S_1,\dots,S_{\ell}$ that we will use to fill the tiles $T_1,\dots,T_{\ell}$. With abuse of notation, we will first define $S_1,\dots, S_{\ell}$ as ordered lists of non-repeated elements and then we will consider the associated sets.

We set the list $S_1$ to be $(g_1,\dots,g_{|T_1|})$.

Given $S_1,\dots, S_{i}$, if $i<\ell$, we define $S_{i+1}$ as follows. Named $g_{j}$ the last element of the list $S_i$, we set
$$S_{i+1}:=(g_{j+1},\dots, g_{j+|T_{i+1}|})$$
where the indexes are considered modulo $v-t$.

Here we note that, since $v-t\geq \max_i(|T_i|)$, the lists $S_1,\dots,S_{\ell}$ do not have repeated elements and hence we can consider them to be sets.

Now we fill, recursively, the tiles $T_1,\dots, T_{\ell}$ with the elements of $S_1,\dots,S_{\ell}$.

For the tile $T_1$ we denote by $\mathcal{R}_1$ the set of the rows such that $R_i\in \mathcal{R}_1$ whenever $(B\setminus T_1)\cap skel(R_i)$ is empty. Set $I_1=\{i\in [1,m]: R_i\in \mathcal{R}_1\}$, we consider a vector $(r_1,\dots,r_m)\in G^m$ such that $r_i=0$ for any $i\in I_1$.
Similarly, we consider the set $\mathcal{C}_1$ of the columns such that $C_j\in \mathcal{C}_1$ whenever $(B\setminus T_1)\cap skel(C_j)$ is empty. Set $J_1=\{j\in [1,n]: C_j\in \mathcal{C}_1\}$, we consider a vector $(c_1,\dots,c_n)\in G^n$ such that $c_j=0$ for any $j\in J_1$.
Since $T_1$ is a nice tile, we can fill it with the elements of $S_1$, obtaining an array $A_1$ such that
\begin{itemize}
\item[$(d_1)$] if the $i$-th row of $A_1$ is non-empty, the sum of its elements, following the natural ordering from left to right, is different from $r_i$;
\item[$(d_2)$] if the $j$-th column of $A_1$ is non-empty, the sum of its elements, following the natural ordering from top to bottom, is different from $c_j$.
\end{itemize}

Now we assume we have defined the arrays $A_1,\dots,A_{b}$. Then, if $b<\ell$, we define the array $A_{b+1}$ as follows. First we denote by $\bar{A}_b$ the union $\bigcup_{i=1}^b A_i$, by $\bar{R}_1,\dots, \bar{R}_m$ its rows and by $\bar{C}_1,\dots, \bar{C}_n$ its columns.
Then we consider the set $\mathcal{R}_{b+1}$ of the rows such that $(B\setminus skel(\bar{A}_b))\cap skel(R_i)$ is non empty but $(B\setminus (skel(\bar{A}_b)\cup T_{b+1}))\cap skel(R_i)$ is empty. Set $I_{b+1}=\{i\in [1,m]: R_i\in \mathcal{R}_{b+1}\}$, and given $i\in I_{b+1}$, we have that, because of property $(c_1)$ of Definition \ref{nice}, the cells of $B\cap skel(R_i)$ are either of the form
$$\{(i,\alpha_{h_{i_1}'+1}),\dots,(i,\alpha_{h_i'}),(i,\beta_i),(i,\beta_i+1),\dots,(i,\beta_i+(h_i-1)),(i,\alpha_1),\dots,(i,\alpha_{h_{i_1}'})\}$$
or
$$\{(i,\beta_i+h_{i_1}),\dots,(i,\beta_i+(h_i-1)),(i,\alpha_1),(i,\alpha_2),\dots,(i,\alpha_{h_i'}),(i,\beta_i),\dots,(i,\beta_i+(h_{i_1}-1))\}$$
where the $\alpha$'s appears in the cells of $skel(\bar{A}_b)$ and the $\beta$'s in those of $T_{b+1}$.
In the following we will denote by $\bar{a}_{i,j}$ the element of $\bar{A}_b$ in position $(i,j)$.
Here, given $i\in I_{b+1}$, we consider a vector $(r_1,\dots,r_m)\in G^m$ such that $r_i=-\sum_{j=1}^{h_i'} \bar{a}_{(i,\alpha_j)}$ where the sum is taken, starting from the cell $(i,\alpha_1)$, following the natural ordering from left to right.

Similarly, we consider the set $\mathcal{C}_{b+1}$ of the columns such that $(B\setminus skel(\bar{A}_b))\cap skel(C_j)$ is non empty but $(B\setminus (skel(\bar{A}_b)\cup T_{b+1}))\cap skel(C_j)$ is empty. Set $J_{b+1}=\{j\in [1,n]: C_j\in \mathcal{C}_{b+1}\}$, and given $j\in J_{b+1}$, we have that, because of property $(c_2)$ of Definition \ref{nice}, the filled cells of $skel(C_j)$ are either of the form
$$\{(\alpha_{k_{j_1}'+1},j),\dots,(\alpha_{k_j'},j),(\gamma_j,j),(\gamma_j+1,j),\dots,(\gamma_j+(k_j-1),j),(\alpha_1,j),\dots,(\alpha_{k_{j_1}'},j)\}$$
or
$$\{(\gamma_j+k_{j_1},j),\dots,(\gamma_j+(k_j-1),j),(\alpha_1,j),(\alpha_2,j),\dots,(\alpha_{k_j'},j),(\gamma_j,j),\dots,(\gamma_j+(k_{j_1}-1),j)\}$$
where the $\alpha$'s appears in the cells of $skel(\bar{A}_b)$ and the $\gamma$'s in those of $T_{b+1}$.
Here, given $j\in J_{b+1}$, we consider a vector $(c_1,\dots,c_n)\in G^n$ such that $c_j=-\sum_{i=1}^{k_j'} \bar{a}_{(\alpha_i,j)}$ where the sum is taken, starting from the cell $(\alpha_1,j)$, following the natural ordering from top to bottom.
Since $T_{b+1}$ is a nice tile, we can fill it with the elements of $S_{b+1}$, obtaining an array $A_{b+1}$ such that
\begin{itemize}
\item[$(d_1)$] if the $i$-th row of $A_{b+1}$ is non-empty, the sum of its elements, following the natural ordering from left to right, is different from $r_i$;
\item[$(d_2)$] if the $j$-th column of $A_{b+1}$ is non-empty, the sum of its elements, following the natural ordering from top to bottom, is different from $c_j$.
\end{itemize}
Here we note that, given a row $R_i\in \mathcal{R}_{b+1}$ with $i\in I_{b+1}$, because of property $(c_1)$ of Definition \ref{nice},  $\E(R_i\cap \bar{A}_{b+1})$ is, following the natural ordering, either of the form
$$(\bar{a}_{h_{i_1}'+1},\dots,\bar{a}_{h_i'},a_1,a_2,\dots,a_{h_i},\bar{a}_1,\dots,\bar{a}_{h_{i_1}'})$$
or
$$(a_{h_{i_1}+1},\dots,a_{h_i},\bar{a}_1,\bar{a}_2,\dots,\bar{a}_{h_i'},a_1,\dots,a_{h_{i_1}})$$
where the $\bar{a}$'s belong to $\bar{A_b}$ and the $a$'s are elements of $A_{b+1}$.

In both cases, we have that
$$a_1+\dots+a_{h_i}\not=-(\bar{a}_{1}+\dots+\bar{a}_{h_i'})=-\bar{a}_{h_i'}-\dots-\bar{a}_{1}.$$
This implies, in the first case that
$$\bar{a}_{h_{i_1}'+1}+\dots + \bar{a}_{h_i'}+a_1+a_2+\dots+a_{h_i}+\bar{a}_1+\dots+\bar{a}_{h_{i_1}'}\not=0$$
and, in the latter, that
$$a_{h_{i_1}+1}+\dots+a_{h_i}+\bar{a}_1+\bar{a}_2+\dots+\bar{a}_{h_i'}+a_1+\dots+a_{h_{i_1}}\not=0.$$
{This means that the sum of the elements of $R_i$ is, in both cases, non-zero for any $i\in I_{b+1}$.}

Reasoning in the same way also for the columns, we obtain that also the sum of the elements of $C_j$ is non-zero whenever $j\in J_{b+1}$.

Therefore, we have that $\bigcup_{i=1}^{\ell} A_i$ defines a $^\lambda \N\mathrm{H}_t(m,n;h,k)$ over $G$ relative to $J$.

Finally, we assume that $\lambda$ is odd. Here, since the necessary conditions are satisfied, we do not have involutions in $G\setminus J$ and we note that if $x\in G\setminus J$ then also $-x$ is in $G\setminus J$. It follows that $G\setminus J$ is even and we can enumerate the $v-t$ elements of $G\setminus J$ as $g_1,\dots,g_{v-t}$ with the additional property that if $x\in \{g_1,\dots,g_{\frac{v-t}{2}}\}$ then $-x\in \{g_{\frac{v-t}{2}+1},\dots,g_{v-t}\}$.

Here we note that
$$\Omega=(G\setminus J)^{(\lambda-1)/2}\cup \{g_1,\dots,g_{\frac{v-t}{2}}\}$$
is such that $\pm \Omega=(G\setminus J)^{\lambda}$. Moreover, proceeding as in the case $\lambda$ even we have that
$$\left(\frac{\lambda-1}{2}|G\setminus J|\right)+\frac{|G\setminus J|}{2}=\frac{\lambda}{2}(v-t)=nk=\sum_{i=1}^{\ell} |T_i|.$$
It follows that we can partition $\Omega$ with lists $S_1,\dots,S_{\ell}$ that do not have repeated elements and such that $|S_i|=|T_i|$.
Then, proceeding as in the case $\lambda$ even, we obtain the thesis also in the odd case.
\endproof
\begin{rem}
The procedure of the previous theorem reduces the dependencies among the events $X_i$'s (resp.  $Y_i$'s) and hence the edges in the dependencies graph (see \cite{LLL}).
Indeed the event $X_i$, that the $i$-th row sums to zero, only depends on the tiles that intersect $R_i$.

Here our first aim was to apply the Lov\'asz Local Lemma (see \cite{LLL}) to solve some instances of the existence problem. 
However, since we only use nice tiles, even these dependencies do not affect the procedure of Theorem \ref{tiling}, so we can fill the array in a very general context, by simply acting locally, tile by tile.
\end{rem}
\section{Totally Filled Arrays}
The goal of this section is to prove that, if we consider a totally filled $m\times n$ rectangular array $B$, we can define a $^\lambda \N\mathrm{H}_t(m,n)$ over $G$ relative to $J$ whose cells are those of $B$ whenever the necessary conditions of Section \ref{necessarySection} are satisfied and $|G|=v\geq 30$.
First of all, we consider the case $m=1$.
\begin{prop}\label{m1}
Let $G$ be a group of size $v$ and let $J$ be a subgroup of $G$ of size $t$. Then there exists a $^\lambda \N\mathrm{H}_t(1,n)$ over $G$ relative to $J$ whenever the necessary conditions of Section \ref{necessarySection} are satisfied.
\end{prop}
\proof
We may assume $G$ is not isomorphic to $\mathbb{Z}_2^r$ since this case has been already considered in Theorem \ref{z2h}. This means that $G$ is either non-abelian or isomorphic to $H\times \mathbb{Z}_2^{r_1}$ where $H$ is a non-trivial abelian group of odd order and $r_1\geq 0$.

CASE 1: $G$ is non-abelian. We claim that, in this case, we have two elements $x,y$ in $G\setminus J$ such that $x+y\not=y+x$. We first note that, since $G$ is non-abelian the center $Z(G)$ of $G$ has a size of at most $|G|/2$. Since $G\setminus J$ has a size of at least $|G|/2$ and it does not contain the identity, there is $y\in G\setminus J$ that is not in the center of $G$.
Here we have that the centralizer $Z(y)$ of $y$ is a proper subgroup of $J$ and has, again, size at most $|G|/2$. Therefore, there exists $x\in G\setminus J$ that is not in the centralizer of $y$ and, in particular $x+y \neq y+x$ and $x\not=\pm y$.

Now we order the elements of $G\setminus J$ as $g_1,\dots,g_{v-t}$ so that $g_1 = x$ and $g_2 = y$. Moreover, if $\lambda$ is odd, we have that $G\setminus J$ does not contain any involution and thus we can assume that $\pm\{g_1,\dots,g_{\frac{v-t}{2}}\}=G\setminus J$ as done in the proof of Theorem \ref{tiling}.
In the following, we denote by $B$ the set of cells of an $1\times n$ totally filled array.
We fill the elements of $B$ so that, in position $(1,i)$ we put the element $g_i$ where the index $i$ is considered modulo $v-t$. Denoted by $A_1$ the array so defined, we have that its element in position $(1,i)$ is $a_{1,i}=g_{i\pmod{v-t}}$.

Since the necessary conditions are satisfied, we have filled $B$ with the elements of:
$$\Omega:=\begin{cases}(G\setminus J)^{\lambda/2}\mbox{ if }\lambda\equiv 0\pmod{2};\\(G\setminus J)^{(\lambda-1)/2}\cup \{g_1,\dots,g_{\frac{v-t}{2}}\}\mbox{ otherwise}.
\end{cases}$$
In both cases, we have that the elements of $\pm\E(A_1)=\pm \Omega$ cover $G\setminus J$ exactly $\lambda$ times.

Clearly, the sum over every column of $A_1$ is non-zero so it is left to prove that we can impose the sum over the first row to be non-zero.
Let us assume that
\begin{equation}\label{xy}
\sum_{a\in \E(R_1)} a= x+y+\sum_{i=3}^n a_{1,i}=0.
\end{equation}
Indeed, if this equation does not hold, $A_1$ is already a $^\lambda \N\mathrm{H}_t(1,n)$ over $G$ relative to $J$.
If equation \eqref{xy} holds, since $x+y\not=y+x$ we must have that
$$y+x+\sum_{i=3}^n a_{1,i}\not=0.$$
Therefore, in this case, it suffices to switch $x$ and $y$ to obtain an array $A$ that is a $^\lambda \N\mathrm{H}_t(1,n)$ over $G$ relative to $J$.

CASE 2: $G=H\times \mathbb{Z}_2^{r_1}$ where $H$ is a non-trivial abelian group of odd order.
In this case we have that $G\setminus (\{0_H\}\times \mathbb{Z}_2^{r_1})$ does not contain involutions and has size at least $\frac{2|G|}{3}>|G|/2$. Since $G\setminus J$ has size at least $|G|/2$, there exists a non involution element $x$ in $G\setminus J$.

Now we order the elements of $G\setminus J$ as we did in CASE 1 and we assume that $g_1 = x$. Also here we denote by $B$ the set of cells of an $1\times n$ totally filled array and we fill the elements of $B$ so that, in position $(1,i)$ we put the element $g_i$ where the index $i$ is considered modulo $v-t$. Denoted by $A_1$ the array so defined, we have that its element in position $(1,i)$ is $a_{1,i}=g_{i\pmod{v-t}}$.
Also in this case, we have that $\pm\E(A_1)$ covers $G\setminus J$ exactly $\lambda$ times.

Clearly, the sum over every column of $A_1$ is non-zero so it is left to prove that we can impose the sum over the first row to be non-zero.
Let us assume that
\begin{equation}\label{x}
\sum_{a\in \E(R_1)} a= x+\sum_{i=2}^n a_{1,i}=0.
\end{equation}
Indeed, if this equation does not hold, $A_1$ is already a $^\lambda \N\mathrm{H}_t(1,n)$ over $G$ relative to $J$.
If equation \eqref{x} holds, since $-x\not=x$ we must have that
$$-x+\sum_{i=2}^n a_{1,i}\not=0.$$
Therefore, in this case, it suffices to change $x$ with $-x$ to obtain an array $A$ that is a $^\lambda \N\mathrm{H}_t(1,n)$ over $G$ relative to $J$.
\endproof
Now we list two families of nice tiles that we will use as bricks to fill the rectangular arrays.
\begin{prop}\label{3b}
Let $T$ be a set of cells of an $m\times n$ array that satisfies the following conditions:
\begin{itemize}
\item[1)] $T$ has exactly $3$ non-empty rows $R_{\alpha},R_{\alpha+1},R_{\alpha+2}$ where the sum are considered modulo $m$;
\item[2)] Each row $R_{\alpha},R_{\alpha+1},R_{\alpha+2}$ is non-empty exactly in the cells of the $b\geq 3$ columns $C_{\beta},C_{\beta+1},\dots,C_{\beta+(b-1)}$ where the sum are considered modulo $n$.
\end{itemize}
Then $T$ is nice.
\end{prop}
\proof
Let us consider a set $S$ of elements of a finite group $G$ and vectors $(r_1,\dots, r_m)\in G^m$, $(c_1,\dots, c_n)\in G^n$.
Now we choose, uniformly at random, an array $A$ such that $skel(A)=T$ and $\E(A)=S$.  Following the notation of Theorem \ref{th:ExistenceLambdaCyclic},  we denote by $\mathbb{P}(X_i)$ the probability of the event $X_i$ that the $i$-th row sums to $r_i$.
Here, if $i\in\{\alpha,\alpha+1,\alpha+2\}$, we have that, chosen the first $b-1$ elements (following the natural ordering) $a_{i,\beta},\dots, a_{i,\beta+(b-2)}$ of the $i$-th row, there is at most one element $x\in S$ such that
$$a_{i,\beta}+\dots+ a_{i,\beta+(b-2)}+x=r_i.$$
It follows that, if $i\in\{\alpha,\alpha+1,\alpha+2\}$,
$$\mathbb{P}(X_i)\leq \frac{1}{|S|-(b-1)}=\frac{1}{2b+1}.$$
Now we denote by $\mathbb{E}(X)$ the expected value of the random variable $X$ given by the number of rows $R_i$, with $i\in\{\alpha,\alpha+1,\alpha+2\}$, that sums to $r_i$.
Due to the linearity of the expected value, we have that:
$$\mathbb{E}(X)=\mathbb{P}(X_{\alpha})+\mathbb{P}(X_{\alpha+1})+\mathbb{P}(X_{\alpha+2})\leq \frac{3}{2b+1}<\frac{1}{2}.$$
Similarly, if we denote by $\mathbb{E}(Y)$ the expected value of the random variable $Y$ given by the number of columns $C_j$, with $j\in\{\beta,\beta+1,\dots,\beta+(b-1)\}$, that sums to $c_j$, we have that:
$$\mathbb{E}(Y)\leq \frac{b}{3b-2}<\frac{1}{2}.$$
Since
$$\mathbb{E}(X)+\mathbb{E}(Y)<1$$
there exists an array $A$ whose skeleton is $T$ and such that $\E(A)=S$ and:
\begin{itemize}
\item[$(d_1)$] if the $i$-th row of $A$ is non-empty, the sum of its elements is different from $r_i$;
\item[$(d_2)$] if the $j$-th column of $A$ is non-empty, the sum of its elements is different from $c_j$.
\end{itemize}
It follows that $T$ is a nice tile.
\endproof
\begin{ex}
Here we show an example of a nice tile $T$, $|T| = 12$, of an $m \times n$ array with $m =5$ and $n \geq 6$, that has $3$ non-empty rows and $4$ non-empty columns. The cells that belong to $T$ are represented by a $\bullet$.
\begin{center}
$\begin{array}{|r|r|r|r|r|r|r|r|r|}
\hline \bullet & \bullet & \:\: & \cdots & \:\: & \bullet & \bullet \\
\hline \bullet & \bullet & & \cdots & & \bullet & \bullet \\
\hline & & & \cdots & & & \\
\hline & & & \cdots & & & \\
\hline \bullet & \bullet & & \cdots & & \bullet & \bullet \\
\hline
\end{array}$
\end{center}
\end{ex}
\begin{prop}\label{2b}
Let $T$ be a set of cells of an $m\times n$ array that satisfies the following conditions:
\begin{itemize}
\item[1)] $T$ has exactly $2$ non-empty rows $R_{\alpha},R_{\alpha+1}$ where the sum is considered modulo $m$;
\item[2)] Each row $R_{\alpha},R_{\alpha+1}$ is non-empty exactly in the cells of the $b\geq 4$ columns $C_{\beta},C_{\beta+1},\dots,C_{\beta+(b-1)}$ where the sum is considered modulo $n$.
\end{itemize}
Then $T$ is nice.
\end{prop}
\proof
Let us consider a set $S$ of elements of a finite group $G$ and vectors $(r_1,\dots, r_m)\in G^m$, $(c_1,\dots, c_n)\in G^n$.
We proceed as in the proof of Proposition \ref{3b} by choosing, uniformly at random an array $A$ such that $skel(A)=T$ and $\E(A)=S$.

Here, denoted by $\mathbb{E}(X)$ the expected value of the random variable $X$ given by the number of rows $R_i$, with $i\in\{\alpha,\alpha+1\}$, that sums to $r_i$,
we have that:
$$\mathbb{E}(X)\leq \frac{2}{2b-(b-1)}=\frac{2}{b+1}.$$
Similarly, if we denote by $\mathbb{E}(Y)$ the expected value of the random variable $Y$ given by the number of columns $C_j$, with $j\in\{\beta,\beta+1,\dots,\beta+(b-1)\}$, that sums to $c_j$, we have that:
$$\mathbb{E}(Y)\leq \frac{b}{2b-1}.$$
Since $b\geq 4$, we have that
$$\mathbb{E}(X)+\mathbb{E}(Y)\leq \frac{2}{b+1}+\frac{b}{2b-1}<1.$$
Therefore there exists an array $A$ whose skeleton is $T$ and such that $\E(A)=S$ and:
\begin{itemize}
\item[$(d_1)$] if the $i$-th row of $A$ is non-empty, the sum of its elements is different from $r_i$;
\item[$(d_2)$] if the $j$-th column of $A$ is non-empty, the sum of its elements is different from $c_j$.
\end{itemize}
It follows that $T$ is a nice tile.
\endproof
\begin{ex}
Here we show an example of a nice tile $T$, $|T| = 10$, of an $m \times n$ array with $m =4$ and $n \geq 7$, that has $2$ non-empty rows and $5$ non-empty columns.
\begin{center}
$\begin{array}{|r|r|r|r|r|r|r|r|r|r|r|}
\hline \bullet & \bullet & \:\: & \cdots & \:\: & \bullet & \bullet & \bullet\\
\hline & & & \cdots & & & & \\
\hline \bullet & \bullet & & \cdots & & \bullet & \bullet & \bullet \\
\hline
\end{array}$
\end{center}
\end{ex}
Using the nice tiles of Propositions \ref{3b} and \ref{2b}, we can prove that
\begin{thm}\label{rectangular}
Let $G$ be a group of size $v$ and let $J$ be a subgroup of $G$ of size $t$. Then there exists a $^\lambda \N\mathrm{H}_t(m,n)$ over $G$ relative to $J$ whenever the necessary conditions of Section \ref{necessarySection} are satisfied and $v-t\geq 15$.
\end{thm}
\proof
We denote by $B$ the set of cells of an $m\times n$ totally filled array and we assume, without loss of generality, that $m\leq n$.

In the following we will assume that $m>1$ since the case $m=1$ has already been considered in Proposition \ref{m1}.
Since $m>1$, we can partition the rows of $B$ into sets $\mathcal{H}_1,\dots,\mathcal{H}_{\ell}$ of consecutive rows such that each $\mathcal{H}_i$ contains either two or three rows. Here we will say that $\mathcal{H}_i$ has weight, respectively $2$ or $3$.

Now, if $n\geq 4$, we can partition the cells that belong to each set $\mathcal{H}_i$ of weight $2$ with tiles defined in Proposition \ref{2b}: it suffices to use tiles with $b\in \{4,5,6,7\}$ (i.e. that has $4,5,6$ or $7$ non-empty consecutive columns).

Similarly, if $n\geq 3$, we can partition the cells that belong to each set $\mathcal{H}_i$ of weight $3$ with tiles defined in Proposition \ref{3b}: it suffices to use tiles with $b\in \{3,4,5\}$ (i.e. that has $3,4$ or $5$ non-empty consecutive columns).

Summing up, if $n\geq 4$, we have that:
\begin{itemize}
\item[1)] it is possible to partition $B$ into nice tiles $T_1,\dots,T_{\ell}$;
\item[2)]$\max_i(|T_i|)\leq 15.$
\end{itemize}
Therefore, because of Theorem \ref{tiling}, there exists a $^\lambda \N\mathrm{H}_t(m,n)$ over $G$ relative to $J$ whenever the necessary conditions are satisfied, $n\geq 4$ and $v-t\geq 15$. It follows that the thesis is proved whenever $n\geq 4$.

Now we assume that $n< 4$ which implies that either $n=m=3$ or $mn\leq 6$. In the first case, if $n=m=3$,  $B$ itself is nice, due to Proposition \ref{3b}. Therefore, because of Theorem \ref{tiling}, there exists a $^\lambda \N\mathrm{H}_t(3,3)$ over $G$ relative to $J$ whenever the necessary conditions are satisfied and $v-t \geq 15 > 9$.

Finally, let us suppose that $mn\leq 6$, we have that $v-t\leq 12$ that is in contradiction with the hypothesis that $v-t\geq 15$.

It follows that, there exists a $^\lambda \N\mathrm{H}_t(m,n)$ over $G$ relative to $J$ whenever the necessary conditions are satisfied and $v-t\geq 15$.
\endproof
We remark that, for any group $G$ whose order $v$ is odd we have that $t\leq v/3$ while, if the order is even $t\leq v/2$ where $t$ is the order of a subgroup $J$. It follows that
\begin{cor}\label{resultonv}
Let $G$ be a group of size $v$ and let $J$ be a subgroup of $G$ of size $t$. Then there exists a $^\lambda \N\mathrm{H}_t(m,n)$ over $G$ relative to $J$ whenever the necessary conditions of Section \ref{necessarySection} are satisfied and $v$ is an even integer $v\geq 30$ or $v$ is odd and such that $v\geq 23$.
\end{cor}
\section{Partially Filled Arrays}
In this section, we will consider the case of arrays that are not totally filled. Due to the necessary conditions of Section \ref{necessarySection}, if a $^\lambda \N\mathrm{H}_t(m,n;h,k)$ exists, we must have that $nk=mh$ that implies that $\lcm(m,n)|(nk)$. In the following we will provide a construction of a $^\lambda \N\mathrm{H}_t(m,n;h,k)$ according to whether $r=\frac{nk}{\lcm(m,n)}$ is $1,2$ or at least $3$.
\subsection{Case $r=1$}
\begin{prop}\label{h=1}
Let $G$ be a group of size $v$ and let $J$ be a subgroup of $G$ of size $t$. Then there exists a $^\lambda \N\mathrm{H}_t(m,n;1,k)$ over $G$ relative to $J$ assuming that the necessary conditions of Section \ref{necessarySection} are satisfied and $v-t\geq 4$.
\end{prop}
\proof
Note that due to the necessary condition $m h = n k$ and $h=1$ we have that $n = m/k$.  Now we denote by $Q$ the set of cells given by
$$Q:=\{(i,1):\ 1\leq i\leq k\}.$$
We consider the subset of the $m\times n$ array defined by:
$$B:=\bigcup_{i=0}^{m/k-1} Q+i(k,1).$$

We order the elements of $G\setminus J$ as $g_1,\dots,g_{v-t}$ in such a way that $g_i+g_{i+1}\not=0$ for any $i\in \{1,\dots,v-t\}$ where the sum is considered modulo $v-t$. Note that this is possible since $v-t\geq 4$. Moreover, if $\lambda$ is odd, we have that $G\setminus J$ does not contain any involution and thus we can assume that $\pm\{g_1,\dots,g_{\frac{v-t}{2}}\}=G\setminus J$  as done in the proof of Theorem \ref{tiling}.

Now we fill the elements of $B$ so that, in the $i$-th row we put the element $g_i$ where the index $i$ is considered modulo $v-t$. We denote by $A_1$ the array so defined.  Since the necessary conditions are satisfied, we have filled $B$ with the elements of:
$$\Omega:=\begin{cases}(G\setminus J)^{\lambda/2}\mbox{ if }\lambda\equiv 0\pmod{2};\\(G\setminus J)^{(\lambda-1)/2}\cup \{g_1,\dots,g_{\frac{v-t}{2}}\}\mbox{ otherwise}.
\end{cases}$$
In both cases, we have that $\pm\E(A_1)=\pm \Omega$ cover $G\setminus J$ exactly $\lambda$ times.  We note that the filled cells in the $i$-th column are those of $Q+(i-1)(k,1)$. Here, if $k=1$ or $k=2$ we obtain a $^\lambda \N\mathrm{H}_t(m,n;1,k)$ since, for any $j\in \{1,\dots,v-t\}$, both $g_{j}$ and $g_{j}+g_{j+1}$ are nonzero.

Let now assume $k\geq 3$ and $n\geq 2$ since the case $n=1$ (or $m=1$) has already been considered in Proposition \ref{m1}. We consider the first column $C_i$,  whose elements are $g_j,  g_{j+1},  \ldots, g_{j+k-1}$ where $j = (i-1)k$,  that sums to zero.  Here the next column is $C_{i+1}$ whose elements are $g_{j+k},g_{j+k+1},\dots, g_{j+2k-1}$.  We note that, since $v-t\geq 4$,  $g_{j+k-1} \not \in \{g_{j+k},  g_{j+k+1} \}$ and hence given $x \in \{ g_{j+k},  g_{j+k+1} \}$ we have that
$$
	g_j + g_{j+1} + \ldots + g_{j+k-2} + x \neq 0.
$$
Let now assume that
\begin{equation}\label{eq:scambio1}
	g_{j+k-1} + g_{j+k+1} + g_{j+k+2} +  \ldots + g_{j+2k-1} = 0,
\end{equation}
since otherwise we can interchange the last element of $C_i$ and the first of $C_{i+1}$ to ensure that both these columns do not sum to zero.  From equation \eqref{eq:scambio1} it follows that
$$
	g_{j+k-1} + g_{j+k} + g_{j+k+2} +  \ldots + g_{j+2k-1} \neq 0.
$$
Therefore,  also in this case,  we can ensure that the columns $C_i$ and $C_{i+1}$ do not sum to zero by permuting the elements $\{g_{j+k-1},  g_{j+k},  g_{j+k+1}\}$ following the cycle $(g_{j+k-1},  g_{j+k}, $ $g_{j+k+1})$. 

Note that this procedure can be reiterated since it involves the last element of $C_i$, the first two elements of $C_{i+1}$ and since $k \geq 3$.  Hence,  we eventually obtain a $^\lambda \N\mathrm{H}_t(m,n;1,k)$ also when $k\geq 3$.
\endproof
\begin{defi}
We name by $Q$ the set of cells of an $m\times n$ array given by
$$Q:=\{(i,j):\ 1\leq i\leq k; \ 1\leq j\leq h\}.$$
Then, given $b$ such that $bk\leq m$ and $bh\leq n$, the set of cells given by
$$T:=\bigcup_{i=0}^{b-1} Q+i(k,h)$$
is said to be an $(k,h)$-stair of length $b$.
\end{defi}

\begin{prop}
A $(3,2)$-stair of length $b$, $T$, is nice whenever $b\geq 2$.
\end{prop}
\proof
Let us consider a set $S$ of elements of a group $G$ and vectors $(r_1,\dots, r_m)\in G^m$, $(c_1,\dots, c_n)\in G^n$.
We note that $T$ has exactly $3b$ rows each with $2$ non-empty cells. Here we consider a non-empty row $R_i$ and we denote by $(i,\beta),(i,\beta+1)$ its non-empty cells where the sum is considered modulo $n$.
Now we chose, uniformly at random, an array $A$ such that $skel(A)=T$ and $\E(A)=S$ and we denote by $\mathbb{P}(X_i)$ the probability of the event $X_i$ that the $i$-th row sums, following the natural ordering, to $r_i$.
Here we have that, chosen the first element $a_{i,\beta}$ of the $i$-th row there is at most one element $x\in S$ such that
$$a_{i,\beta}+x=r_i.$$
It follows that
$$\mathbb{P}(X_i)\leq \frac{1}{|S|-1}=\frac{1}{6b-1}.$$
Now we denote by $\mathbb{E}(X)$ the expected value of the random variable $X$ given by the number of non-empty rows $R_i$ that sums to $r_i$.
Due to the linearity of the expected value, we have that:
$$\mathbb{E}(X)=\sum_{i:\ skel(R_i)\cap T\not=\emptyset}\mathbb{P}(X_{i})\leq \frac{3b}{6b-1}.$$

Similarly, given a non-empty column $C_j$, we denote by $\mathbb{P}(Y_j)$ the probability of the event $Y_j$ that $C_j$ sums to $c_j$.
Proceeding as we did for the rows, and noting that the number of the elements in the $j$-th column is three, we have that
$$\mathbb{P}(Y_i)\leq \frac{1}{|S|-2}=\frac{1}{6b-2}.$$
Therefore, using again the linearity of the expected value, and denoting by $\mathbb{E}(Y)$ the expected value of the random variable $Y$ given by the number of non-empty columns $C_j$ that sums to $c_j$, we have that:
$$\mathbb{E}(Y)=\sum_{j:\ skel(C_j)\cap T\not=\emptyset}\mathbb{P}(X_{j})\leq \frac{2b}{6b-2}.$$
It follows that
$$\mathbb{E}(X)+\mathbb{E}(Y)\leq \frac{3b}{6b-1}+\frac{2b}{6b-2}< \frac{5b}{6b-2}.$$
Since we are assuming that $b\geq 2$, it follows that $$\frac{5b}{6b-2}\leq 1.$$
Therefore, there exists an assignment of the values of the cells of $T$ among $S$ that satisfies the conditions of Definition \ref{nice}.
\endproof

\begin{ex}
Here we show an example of a $(3,2)$-stair of length $3$, $T$, $|T| = 18$, of an $m \times n$ array with $m=9$ and $n=6$.
\begin{center}
$\begin{array}{|r|r|r|r|r|r|}
\hline \bullet & \bullet & \:\: & \:\: & \:\: & \:\:\\
\hline \bullet & \bullet & \:\: & & \:\: & \:\:\\
\hline \bullet & \bullet & \:\: & & \:\: & \:\:\\

\hline \:\: & \:\: & \bullet & \bullet & \:\: & \:\:\\
\hline \:\: & \:\: & \bullet & \bullet & \:\: & \:\:\\
\hline \:\: & \:\: & \bullet & \bullet & \:\: & \:\:\\

\hline \:\: & \:\: & \:\: & \:\: & \bullet & \bullet \\
\hline \:\: & \:\: & \:\: & \:\: & \bullet & \bullet \\
\hline \:\: & \:\: & \:\: & \:\: & \bullet & \bullet \\

\hline
\end{array}$
\end{center}
\end{ex}

\begin{prop}
Let $G$ be a group of size $v$ and let $J$ be a subgroup of $G$ of size $t$. Then there exists a $^\lambda \N\mathrm{H}_t(m,n;h,k)$ over $G$ relative to $J$ assuming that $r=\frac{nk}{\lcm(m,n)}= 1$, the necessary conditions of Section \ref{necessarySection} are satisfied and $v-t\geq 18$.
\end{prop}
\proof
In this case we have that $nk=mh=\lcm(m,n)$. It follows that
$$k= \frac{\lcm(m,n)}{n}$$
and
$$h= \frac{\lcm(m,n)}{m}.$$
We name by $Q$ the set of cells given by
$$Q:=\{(i,j):\ 1\leq i\leq k; \ 1\leq j\leq h\}.$$
Then, since
$$\frac{n}{h}=\frac{m}{k}=\frac{mn}{\lcm(m,n)}=\gcd(m,n),$$
we can consider the subset of the $m\times n$ array defined by:
$$B:=\bigcup_{i=0}^{n/h-1} Q+i(k,h).$$
Note that $k=h$ would imply that $m=n=\lcm(m,n)$ and hence, since, $r=1$ it would be possible only if $k=h=1$.
In this case any array $A$ such that $skel(A)=B$ and $\E(A)$ is $\Omega$ where $$\Omega:=\begin{cases}(G\setminus J)^{\lambda/2}\mbox{ if }\lambda\equiv 0\pmod{2};\\(G\setminus J)^{(\lambda-1)/2}\cup \{g_1,\dots,g_{\frac{v-t}{2}}\}\mbox{ otherwise},
\end{cases}$$
{ is a $^\lambda \N\mathrm{H}_t(m,n;h,k)$ over $G$ relative to $J$. Here, if $\lambda$ is odd, we are assuming that $\pm\{g_1,\dots,g_{\frac{v-t}{2}}\}=G\setminus J$, as done in the proof of Theorem \ref{tiling}.}

We can suppose now that $h<k$.

CASE $h=1$: This case follows from Proposition \ref{h=1}.

CASE $h=2$ and $k=3$: Here we can assume that $(m,n)\not=(3,2)$ since otherwise $B$ would be the skeleton of a totally filled array.  In this case $B$ can be tessellated with tiles of type $(3,2)$-stair of length either two or three.
Since those tiles are nice and their sizes are at most $18$, they define a $^\lambda \N\mathrm{H}_t(m,n;h,k)$ over $G$ relative to $J$.

Otherwise we have that $h\geq 3$ and $k>h$. In this case, we have that $Q$ itself can be tessellated with tiles whose sizes are at most $15$ as done in Theorem \ref{rectangular}. Hence the same can be done, by translation, with $B$.

It follows that, assuming $r=1$, there exists a $^\lambda \N\mathrm{H}_t(m,n;h,k)$ over $G$ relative to $J$ whenever the necessary conditions are satisfied and $v-t\geq 18$.
\endproof
\subsection{Case $r=2$}
\begin{prop}\label{kh2}
Let $G$ be a group of size $v$ and let $J$ be a subgroup of $G$ of size $t$. Then there exists a $^\lambda \N\mathrm{H}_t(n;2)$ over $G$ relative to $J$ assuming that the necessary conditions of Section \ref{necessarySection} are satisfied and $v-t\geq 4$.
\end{prop}
\proof
We name by $Q$ the set of cells given by
$$Q:=\{(1,j):\ 1\leq j\leq 2\}.$$
We consider the subset of the $n\times n$ array defined by:
$$B:=\bigcup_{i=0}^{n-1} Q+i(1,1)$$
where the sum is considered modulo $n$.

We order the elements of $G\setminus J$ as $g_1,\dots,g_{v-t}$ in such a way that $g_i+g_{i+1}\not=0$ for any $i\in \{1,\dots,v-t\}$ where the sum is considered modulo $v-t$. Note that this is possible since $v-t\geq 4$. Moreover, if $\lambda$ is odd, we have that $G\setminus J$ does not contain any involution and thus,  as done in the proof of Theorem \ref{tiling}, we can assume that $\pm\{g_1,\dots,g_{\frac{v-t}{2}}\}=G\setminus J$.

Now we divide the proof according to whether $\lambda$ is even or odd.

CASE 1: $\lambda$ is even. Here we fill the elements of $B$ so that,  we put in the $i$-th row, in order, the elements $g_{2i-1}$ and $g_{2i}$ where the indexes are considered modulo $v-t$. We denote by $A$ the array so defined.

Since the necessary conditions are satisfied, we have that $skel(A)=B$ and $\E(A)=\Omega$ where
$$\Omega=(G\setminus J)^{\lambda/2}.
$$
Moreover, because of the definition and considering the indexes modulo $v-t$, we have that for any $i\in [1,n]$,
$$ \begin{cases}
g_{2i-1}+g_{2i}\not=0;\\
g_{2i}+g_{2i+1}\not=0.
\end{cases}$$
Note that, since $\lambda$ is even, here we have that $\E(C_1)=\{g_1,g_{v-t}\}$.  It follows that $\E(R_i)=\{g_{2i-1},g_{2i}\}$ and $\E(C_{i+1})=\{g_{2i},g_{2i+1}\}$, where the row indexes are considered modulo $n$ and the group ones modulo $v-t$, and hence $A$ is a $^\lambda \N\mathrm{H}_t(n;2)$ over $G$ relative to $J$.

CASE 2: $\lambda$ is odd. Here we fill the elements of $B$ so that, in the $i$-th row we put, in order, the elements $g_{2i-1}$ and $g_{2i}$ where the indexes are considered modulo $\frac{v-t}{2}$. We denote by $A$ the array so defined.

Since the necessary conditions are satisfied, we have that $skel(A)=B$ and $\E(A)=\Omega$ where
$$\Omega=\{g_1,\dots,g_{\frac{v-t}{2}}\}^{\lambda}.
$$
Moreover, because of the definition of $A$, we have that if $x\in \E(A)$, then $-x\not\in \E(A)$.
It follows that the sum over each row and column is non-zero.
Therefore, $A$ is a $^\lambda \N\mathrm{H}_t(n;2)$ over $G$ relative to $J$.
\endproof
\begin{defi}
We name by $Q$ the set of cells of an $m\times n$ array given by
$$Q:=\{(i,j):\ 1\leq i\leq k/2; \ 1\leq j\leq h/2\}.$$
Then, given $b$ such that $bk/2\leq m$ and $bh/2\leq n$, the set of cells given by
$$T:=\left(\bigcup_{i=0}^{b-1} Q+i(k/2,h/2)\right)\bigcup\left(\bigcup_{i=0}^{b-1} Q+(0,h/2)+i(k/2,h/2)\right)$$
is said to be a double $(k/2,h/2)$-stair of length $b$.
\end{defi}

\begin{prop}
A double $(2,1)$-stair of length $b$, $T$, is nice whenever $b\geq 3$.
\end{prop}
\proof
Let us consider a set $S$ of elements of a group $G$ and vectors $(r_1,\dots, r_m)\in G^m$, $(c_1,\dots, c_n)\in G^n$.
We note that $T$ has exactly $2b$ rows each with $2$ non-empty cells. Here we consider a non-empty row $R_i$ and we denote by $(i,\beta),(i,\beta+1)$ its non-empty cells where the sum is considered modulo $n$. Now we chose, uniformly at random, an array $A$ such that $skel(A)=T$ and $\E(A)=S$ and we denote by $\mathbb{P}(X_i)$ the probability of the event $X_i$ that the $i$-th row sums, following the natural ordering, to $r_i$.
Here we have that, chosen the first element $a_{i,\beta}$ of the $i$-th row there is at most one element $x\in S$ such that
$$a_{i,\beta}+x=r_i.$$
It follows that
$$\mathbb{P}(X_i)\leq \frac{1}{|S|-1}=\frac{1}{4b-1}.$$
Now we denote by $\mathbb{E}(X)$ the expected value of the random variable $X$ given by the number of non-empty rows $R_i$ that sums to $r_i$.
Due to the linearity of the expected value, we have that:
$$\mathbb{E}(X)=\sum_{i:\ skel(R_i)\cap T \not=\emptyset}\mathbb{P}(X_{i})\leq \frac{2b}{4b-1}.$$

Similarly, given a non-empty column $C_j$, we denote by $\mathbb{P}(Y_j)$ the probability of the event $Y_i$ that $C_j$ sums to $c_j$.
Proceeding as we did for the rows, and noting that the number of the elements of the $j$-th column is four except when $j=1$ or $j=b+1$ when is two, we have that,  if $j\in \{2,3,\dots, b\}$
$$\mathbb{P}(Y_j)\leq \frac{1}{|S|-3}=\frac{1}{4b-3}$$
and
$$\mathbb{P}(Y_1)=\mathbb{P}(Y_{b+1}) \leq \frac{1}{|S|-1}=\frac{1}{4b-1}.$$

Therefore, using again the linearity of the expected value, and denoting by $\mathbb{E}(Y)$ the expected value of the random variable $Y$ given by the number of non-empty columns $C_j$ that sums to $c_j$, we have that:
$$\mathbb{E}(Y)=\sum_{j:\ skel(C_j)\cap T\not=\emptyset}\mathbb{P}(Y_{j})\leq \frac{b-1}{4b-3}+\frac{2}{4b-1}.$$
It follows that, assuming $b\geq 3$,
$$\mathbb{E}(X)+\mathbb{E}(Y)\leq \frac{2b}{4b-1}+\frac{b-1}{4b-3}+\frac{2}{4b-1}<1.$$
Therefore, there exists an assignment of the values of the cells of $T$ among $S$ that satisfies the conditions of Definition \ref{nice}.
\endproof

With essentially the same proof we also obtain that:
\begin{prop}
A double $(3,1)$-stair of length $b$, $T$, is nice whenever $b\geq 2$.
\end{prop}

\begin{ex}
Here we show an example of a double $(2,1)$-stair of length $3$, $T$, $|T| = 12$, of an $m \times n$ array with $m=6$ and $n=4$.
\begin{center}
$\begin{array}{|r|r|r|r|}
\hline \bullet & \bullet & \:\: & \:\: \\
\hline \bullet & \bullet & \:\: & \\
\hline & \bullet & \bullet & \\
\hline & \bullet & \bullet & \\
\hline & & \bullet & \bullet\\
\hline & & \bullet & \bullet\\
\hline
\end{array}$
\end{center}
\end{ex}

We are now ready to prove the following result.
\begin{prop}
Let $G$ be a group of size $v$ and let $J$ be a subgroup of $G$ of size $t$. Then there exists a $^\lambda \N\mathrm{H}_t(m,n;h,k)$ over $G$ relative to $J$ assuming that $r=\frac{nk}{\lcm(m,n)}=2$, the necessary conditions of Section \ref{necessarySection} are satisfied and $v-t\geq 20$.
\end{prop}
\proof
In this case we have that $nk=mh=2\lcm(m,n)$. It follows that
$$\frac{k}{2}= \frac{\lcm(m,n)}{n}$$
and
$$\frac{h}{2}= \frac{\lcm(m,n)}{m}.$$
We name by $Q$ the set of cells given by
$$Q:=\{(i,j):\ 1\leq i\leq k/2; \ 1\leq j\leq h/2\}.$$
Then, since
$$\frac{2n}{h}=\frac{2m}{k}=\frac{mn}{\lcm(m,n)}=\gcd(m,n),$$
we can consider the subset of the $m\times n$ array defined by:
$$B:=\left(\bigcup_{i=0}^{2n/h-1} Q+i(k/2,h/2)\right)\bigcup \left(\bigcup_{i=0}^{2n/h-1} Q+(0,h/2)+i(k/2,h/2)\right).$$
Note that $k=h$ would imply that $m=n=\lcm(m,n)$ and hence, since, $r=2$ it would be possible only if $k=h=2$. In this case, the thesis follows from Proposition \ref{kh2}.

We can suppose now that $h<k$.

CASE $h=2$ and $k=4$: In this case $m$ is at least $6$ and $B$ can be tessellated with double $(2,1)$-stair of length either three or four or five.
Since those tiles are nice and their sizes are at most $20$, they define a $^\lambda \N\mathrm{H}_t(m,n;h,k)$ over $G$ relative to $J$.

CASE $h=2$ and $k=6$: In this case, $m$ is at least $9$ and $B$ can be tessellated with double $(3,1)$-stair of length either two or three.
Since those tiles are nice and their sizes are at most $18$, they define a $^\lambda \N\mathrm{H}_t(m,n;h,k)$ over $G$ relative to $J$.

CASE $h=2$ and $k\geq 8$: In this case $Q\cup (Q+(0,h/2))$ is a $(k/2) \times 2$ rectangle and, since $k/2\geq 4$, can be tessellated with tiles whose sizes are at most $15$ as done in Theorem \ref{rectangular}.  Hence the same can be done, by translation, with $B$. It follows that, also in this case, we obtain a $^\lambda \N\mathrm{H}_t(m,n;h,k)$.

Otherwise we have that $h\geq 4$ and $k>h$. In this case $Q\cup (Q+(0,h/2))$ is a $(k/2) \times h$ rectangle and, since $k/2\geq 3$ and $h\geq 4$, can be tessellated with tiles whose sizes are at most $15$,  again as done in the proof of Theorem \ref{rectangular}.  Hence the same can be done, by translation, with $B$.

It follows that, assuming $r=2$, there exists a $^\lambda \N\mathrm{H}_t(m,n;h,k)$ over $G$ relative to $J$ whenever the necessary conditions are satisfied and $v-t\geq 20$.
\endproof

\subsection{The general case ($r\geq 3$)}
In this case, we consider the following kind of tiles.
\begin{defi}
Let $m,n$ be two positive integer such that $n\geq m$.
Identified the $m\times n$ array with the elements of $\Z_m\times \Z_n$, and given $a < n, b\leq m$, we say that a set $T$ of cells is an $(a,b)$-diagonal tile if its non-empty cells are exactly the following ones:
$$\{(\bar{i},\bar{j})+(x,x),(\bar{i},\bar{j}+1)+(x,x),\dots, (\bar{i},\bar{j}+(a-1))+(x,x): x\in [0,b-1]\}$$
for some $(\bar{i},\bar{j})\in \Z_m\times \Z_n.$
\end{defi}
\begin{prop}
A $(3,b)$-diagonal tile $T$ is nice whenever $b \geq 4$ and either $b \leq n - 2$ or $b = m$.
\end{prop}
\proof
First of all we note that the property $(c_1)$ of Definition \ref{nice} is always satisfied while property $(c_2)$ holds since either $b \leq n - 2$ or $b = m$.

Let us consider a set $S$ of elements of a group $G$ and vectors $(r_1,\dots, r_m)\in G^m$, $(c_1,\dots, c_n)\in G^n$. 
Since $T$ is a $(3,b)$-diagonal tile,  it has exactly $b$ rows each with $3$ non-empty cells.  We consider a non-empty row $R_i$ and we denote by $(i,\beta),(i,\beta+1),(i,\beta+2)$ its non-empty cells where the sum is considered modulo $n$.  Now we choose, uniformly at random, an array $A$ such that $skel(A)=T$ and $\E(A)=S$ and we denote by $\mathbb{P}(X_i)$ the probability of the event $X_i$ that the $i$-th row sums, following the natural ordering, to $r_i$.
Here we have that, chosen the first $2$ elements $a_{i,\beta}, a_{i,\beta+1}$ of the $i$-th row there is at most one element $x\in S$ such that
$$a_{i,\beta}+ a_{i,\beta+1}+x=r_i.$$
It follows that
$$\mathbb{P}(X_i)\leq \frac{1}{|S|-2}=\frac{1}{3b-2}.$$
Now we denote by $\mathbb{E}(X)$ the expected value of the random variable $X$ given by the number of non-empty rows $R_i$ that sums to $r_i$.
Due to the linearity of the expected value, we have that:
$$\mathbb{E}(X)=\sum_{i:\ skel(R_i) \cap T \not=\emptyset}\mathbb{P}(X_{i})\leq \frac{b}{3b-2}.$$

Similarly, given a non-empty column $C_j$, we denote by $\mathbb{P}(Y_j)$ the probability of the event $Y_i$ that $C_j$ sums to $c_j$.
Proceeding as we did for the rows, denoted by $|C_j|$ the number of the elements of the $j$-th column, we have that
$$\mathbb{P}(Y_i)\leq \frac{1}{|S|-|C_j|+1}=\frac{1}{3b-|C_j|+1}.$$
To estimate the expected value $\mathbb{E}(Y)$ of the random variable $Y$ given by the number of non-empty columns $C_j$ that sums to $c_j$, we divide the discussion into three cases.

CASE 1: $b\leq n-2$.
Here we have that $T$ has exactly $b-2$ columns with $3$ non-empty cells, $2$ columns with $2$ non-empty cells and $2$ columns with one non-empty cell.
Due to the linearity of the expected value, we have that:
$$\mathbb{E}(Y)\leq \frac{b-2}{3b-2}+\frac{2}{3b}+\frac{2}{3b-1}.$$
It follows that
$$\mathbb{E}(X)+\mathbb{E}(Y)\leq\frac{b}{3b-2}+ \frac{b-2}{3b-2}+\frac{2}{3b}+\frac{2}{3b-1}<1.$$
CASE 2: $b=n-1$.
Here we have that $T$ has exactly $b-1$ columns with $3$ non-empty cells and $2$ columns with $2$ non-empty cells.
Due to the linearity of the expected value, we have that:
$$\mathbb{E}(Y)\leq \frac{b-1}{3b-2}+\frac{2}{3b-1}.$$
It follows that
$$\mathbb{E}(X)+\mathbb{E}(Y)\leq\frac{b}{3b-2}+ \frac{b-1}{3b-2}+\frac{2}{3b-1}<1.$$
CASE 3: $b=n$.
Here we have that $T$ has exactly $b$ columns with $3$ non-empty cells.
Due to the linearity of the expected value, we have that:
$$\mathbb{E}(Y)\leq \frac{b}{3b-2}.$$
It follows that
$$\mathbb{E}(X)+\mathbb{E}(Y)\leq\frac{2b}{3b-2}<1.$$
\endproof
\begin{ex}
Here we provide an example of a $(3,b)$-diagonal tile $T$ of an $m \times n$ array with $n \geq m = b = n-2$.
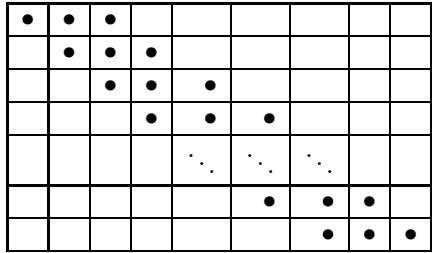
\begin{center}
$\begin{array}{|r|r|r|r|r|r|r|r|r|}
\hline \bullet & \bullet & \bullet & \:\: & \:\: & \:\: & \:\: & \:\: & \:\: \\
\hline \:\: & \bullet & \bullet & \bullet & \:\: & \:\: & \:\: & \:\: & \:\: \\
\hline \:\: & \:\: & \bullet & \bullet & \bullet & \:\: & \:\: & \:\: & \:\: \\
\hline \:\: & \:\: & \:\: & \bullet & \bullet & \bullet & \:\: & \:\: & \:\: \\
\hline \:\: & \:\: & \:\: & \:\: & \ddots & \ddots & \ddots & \:\: & \:\: \\
\hline \:\: & \:\: & \:\: & \:\: & \:\: & \bullet & \bullet & \bullet & \:\: \\
\hline \:\: & \:\: & \:\: & \:\: & \:\: & \:\: & \bullet & \bullet & \bullet \\
\hline
\end{array}$
\end{center}
\end{ex}
With essentially the same proof we also have that:
\begin{prop}
A $(4,b)$-diagonal tile $T$ is nice whenever $b\geq 3$ and either $b \leq n-3$ or $b=m$.
\end{prop}
\begin{ex}
Here we provide an example of a $(4,b)$-diagonal tile $T$ of an $m \times n$ array with $n \geq m = b = n-2$.
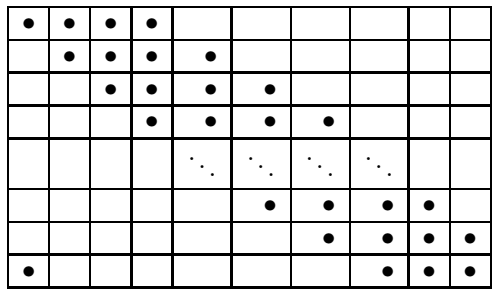
\begin{center}
$\begin{array}{|r|r|r|r|r|r|r|r|r|r|r|}
\hline \bullet & \bullet & \bullet & \bullet & \:\: & \:\: & \:\: & \:\: & \:\:  & \:\: \\
\hline \:\: & \bullet & \bullet & \bullet & \bullet& \:\: & \:\: & \:\: & \:\: & \\
\hline \:\: & \:\: & \bullet & \bullet & \bullet & \bullet & \:\: & \:\: & \:\: & \\
\hline \:\: & \:\: & \:\: & \bullet & \bullet & \bullet & \bullet & \:\: & \:\: & \\
\hline \:\: & \:\: & \:\: & \:\: & \ddots & \ddots & \ddots & \ddots & \:\: & \\
\hline \:\: & \:\: & \:\: & \:\: & \:\: & \bullet & \bullet & \bullet & \bullet  & \\
\hline \:\: & \:\: & \:\: & \:\: & \:\: & \:\: & \bullet & \bullet & \bullet & \bullet \\
\hline \bullet & \:\: & \:\: & \:\: & \:\: & \:\: & \:\: & \bullet & \bullet  & \bullet \\
\hline
\end{array}$
\end{center}
\end{ex}
\begin{prop}
A $(5,b)$-diagonal tile $T$ is nice whenever $b\geq 2$ and either $b \leq n-4$ or $b=m$.
\end{prop}
\begin{ex}
Here we provide an example of a $(5,b)$-diagonal tile $T$ of an $m \times n$ array with $n \geq m = b = n-1$.
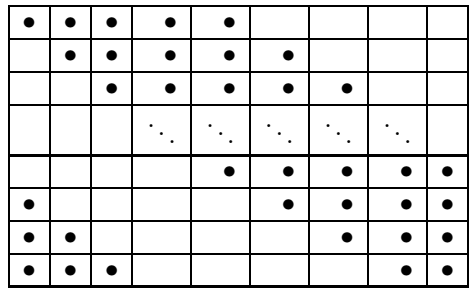
\begin{center}
$\begin{array}{|r|r|r|r|r|r|r|r|r|r|r|}
\hline \bullet & \bullet & \bullet & \bullet & \bullet & \:\: & \:\: & \:\: & \:\: \\
\hline \:\: & \bullet & \bullet & \bullet & \bullet& \bullet & \:\: & \:\: & \:\: \\
\hline \:\: & \:\: & \bullet & \bullet & \bullet & \bullet & \bullet & \:\: & \:\: \\
\hline \:\: & \:\: & \:\: & \ddots & \ddots & \ddots & \ddots & \ddots & \:\: \\
\hline \:\: & \:\: & \:\: & \:\: & \bullet & \bullet & \bullet & \bullet & \bullet \\
\hline \bullet & \:\: & \:\: & \:\: & \:\: & \bullet & \bullet & \bullet & \bullet \\
\hline \bullet & \bullet & \:\: & \:\: & \:\: & \:\: & \bullet & \bullet & \bullet \\
\hline \bullet & \bullet & \bullet & \:\: & \:\: & \:\: & \:\: & \bullet & \bullet \\
\hline
\end{array}$
\end{center}
\end{ex}
\begin{prop}\label{r3}
Let $G$ be a group of size $v$ and let $J$ be a subgroup of $G$ of size $t$. Then there exists a $^\lambda \N\mathrm{H}_t(m,n;h,k)$ over $G$ relative to $J$ assuming that $r=\frac{nk}{\lcm(m,n)}\geq 3$, the necessary conditions of Section \ref{necessarySection} are satisfied and $v-t\geq 21$.
\end{prop}
\proof
We can assume, without loss of generality, that $n\geq m$ and that the array is not totally filled. Here we have that, necessarily,   $n\geq m> r \geq 3$. Indeed, otherwise, since the array is not totally filled, we would have that $k<3$ and hence
$$r=\frac{nk}{\lcm(m,n)}< \frac{3n}{\lcm(m,n)} < 3$$ contradicting the hypothesis that $r\geq 3$.

Assuming now that $n\geq m>3$, $nk=mh$ and $r=\frac{nk}{\lcm(m,n)}\geq 3$, we identify the cells of an $m\times n$ array with the elements of the group $\mathbb{Z}_m\times \Z_n$ and we consider the subgroup $H$ generated by $(1,1)$. As proved in Lemma 4.1 of \cite{CostaDellaFiorePasotti}, $H$ contains exactly $\frac{\lcm(m,n)}{m}$ filled cells in each row
and $\frac{\lcm(m,n)}{n}$ filled cells in each column. Moreover, since the array is not totally filled, $H,H+(0,1),\dots,H+(0,r-1)$ are disjoint cosets.
Now we set
$$B:=\bigcup_{i=0}^{r-1} \{H+(0,i)\}.$$
Then $B$ contains exactly
$$\frac{mh}{\lcm(m,n)}\frac{\lcm(m,n)}{m}=h$$
filled cells in each row and
$$\frac{nk}{\lcm(m,n)}\frac{\lcm(m,n)}{n}=k$$
filled cells in each column.

Since $r\geq 3$, we can partition the family of cosets $\mathcal{H}:=\{H,H+(0,1),\dots,H+(0,r-1)\}$ into subfamilies $\mathcal{H}_1,\dots, \mathcal{H}_{\ell}$ of adjacent ones each of which contains (i.e. it has weight) either $3$, $4$ or $5$ cosets.
This means that each of those families is of the following form
$\mathcal{H}_i= \{H+(0,j),H+(0,j+1),\dots,H+(0,j+w-1)\}$ where $w$ is the weight of $\mathcal{H}_i$ and $j\in [0,r-w]$.

Now, if $w=3$ and since $m\geq 4$,  the cells that belong to $\mathcal{H}_i$ can be tessellated with $(3,b)$-diagonal tiles each of which has $b\in \{4,5,6,7\}$ and $b=m$ when $m \in  \{4,5,6,7\}$ or $b \leq m-4 < n-2$ otherwise.
If instead $w=4$ we have that $m > r \geq 4$.  Then the cells that belong to $\mathcal{H}_i$ can be tessellated with $(4,b)$-diagonal tiles each of which has $b\in \{3,4,5\}$ and $b=m$ when $m=5$ or $b \leq m-3 \leq n-3$ otherwise.
Finally,  if  $w=5$ we have that $m > r \geq 5$.  Then the cells that belong to $\mathcal{H}_i$ can be tessellated with $(5,b)$-diagonal tiles each of which has $b\in \{2,3\}$ and $b \leq m-4 \leq n-4$.

Since all those tiles are nice and their sizes are at most $21$, we have that there exists a $^\lambda \N\mathrm{H}_t(m,n;h,k)$ over $G$ relative to $J$ assuming that $r=\frac{nk}{\lcm(m,n)}\geq 3$, the necessary conditions are satisfied and $v-t\geq 21$.
\endproof
From the result of this section, reasoning as in Corollary \ref{resultonv}, we obtain the main result of this paper.
\begin{thm}\label{thmonv}
Let $G$ be a group of size $v$ and let $J$ be a subgroup of $G$ of size $t$. Then there exists a $^\lambda \N\mathrm{H}_t(m,n;h,k)$ over $G$ relative to $J$ assuming that the necessary conditions of Section \ref{necessarySection} are satisfied and $v$ is an even integer $v\geq 42$ or $v$ is odd and such that $v\geq 29$.
\end{thm}
\section{Application to Biembeddings}
In \cite{A}, Archdeacon introduced Heffter arrays also because
they are useful for finding biembeddings of cycle decompositions, as shown, for instance, in
\cite{CDY, Costa, CMPPHeffter, CPPBiembeddings, CPEJC, CostaPasotti2, DM}.
In this section, generalizing some of his results we show how starting from a $\lambda$-fold non-zero sum Heffter array
it is possible to obtain suitable biembeddings.
In this context, we equip a multigraph $\G$ with the following topology.
\begin{itemize}
\item If $\G$ is a simple graph, $\G$ is viewed with the usual topology as a $1$-dimensional simplicial complex.
\item If $\G$ is a multigraph, we consider the topology naturally induced on $\G$ by a simple graph $\G'$ that is a subdivision of $\G$.
\end{itemize}
Note that this topology is well defined because two multigraphs $\G$ and $\G'$ are homeomorphic if and only if there exists an isomorphism from some subdivision of $\G$ to some subdivision of $\G'$. Now we provide the following definition, see \cite{GT, Moh, MT} for the simple graph case.
\begin{defi}
An \emph{embedding} of a multigraph $\G$ in a surface $\Sigma$ is a continuous injective
mapping $\psi: \G \to \Sigma$, where $\G$ is viewed with the topology described above.
\end{defi}

The connected components of $\Sigma \setminus \psi(\G)$ are called $\psi$-\emph{faces}.
If each $\psi$-face is homeomorphic to an open disc, then the embedding $\psi$ is said to be \emph{cellular}.

\begin{defi}
An embedding $\psi$ of a multigraph $\G$ in a surface $\Sigma$ is said to be a \emph{biembedding} if it is face $2$-colourable.
\end{defi}

Given a $\lambda$-fold non-zero sum Heffter array $^\lambda\N\mathrm{H}_t(m,n; h,k)$, say $A$, and permutations $\omega_{r_1},\dots,\omega_{r_m}$ for the cells of each row, and $\omega_{c_1},\dots,\omega_{c_n}$ for the cells of each column, we define $\omega_r:=\omega_{r_1}\circ \omega_{r_2}\circ \dots\circ \omega_{r_m}$ and $\omega_c:=\omega_{c_1}\circ \omega_{c_2}\circ \dots\circ \omega_{c_n}$. Then we say that $\omega_r$ and $\omega_c$ are \emph{compatible} if $\omega_c \circ \omega_r$ is a cycle of length $|skel(A)|$.

Following the proof of Theorem 5.5 of \cite{CPEJC}, we obtain that:
\begin{thm}\label{thm:biembedding}
Let $A$ be a $\lambda$-fold relative Heffter array $^\lambda\N\mathrm{H}_t(m,n; h,k)$ that admits compatible orderings $\omega_r$ and $\omega_c$.
Then there exists a cellular biembedding of $^\lambda K_{(\frac{2nk}{\lambda t}+1)\times t}$ into an orientable surface whose face lengths are multiples of $k$ and $h$ strictly larger than $k$ and $h$ respectively.
\end{thm}

As already remarked in \cite{CPPBiembeddings},
looking for compatible orderings led us to investigate the following problem
introduced in \cite{CDP}.
Let $A$ be an $m\times n$ \emph{toroidal} p.f. array. By $r_i$ we denote the orientation of the $i$-th row,
precisely $r_i=1$ if it is from left to right and $r_i=-1$ if it is from right to left. Analogously, for the $j$-th
column, if its orientation $c_j$ is from top to bottom then $c_j=1$ otherwise $c_j=-1$. Assume that an orientation
$\R=(r_1,\dots,r_m)$
and $\C=(c_1,\dots,c_n)$ is fixed. Given an initial filled cell $(i_1,j_1)$ consider the sequence
$ L_{\R,\C}(i_1,j_1)=((i_1,j_1),(i_2,j_2),\ldots,(i_\ell,j_\ell),$ $(i_{\ell+1},j_{\ell+1}),\ldots)$
where $j_{\ell+1}$ is the column index of the filled cell $(i_\ell,j_{\ell+1})$ of the row $R_{i_\ell}$ next
to
$(i_\ell,j_\ell)$ in the orientation $r_{i_\ell}$,
and where $i_{\ell+1}$ is the row index of the filled cell of the column $C_{j_{\ell+1}}$ next to
$(i_\ell,j_{\ell+1})$ in the orientation $c_{j_{\ell+1}}$.
The problem is the following:

\begin{KN}
Given a toroidal p.f. array $A$,
do there exist $\R$ and $\C$ such that the list $L_{\R,\C}$ covers all the filled
cells of $A$?
\end{KN}

By $P(A)$ we will denote the \probname\ for a given array $A$.
Also, given a filled cell $(i,j)$, if $L_{\R,\C}(i,j)$ covers all the filled positions of $A$ we will
say that $(\R,\C)$ is a solution of $P(A)$.
For known results about this problem see \cite{CDP}.
The relationship between the Crazy Knight's Tour Problem and $\lambda$-fold non-zero sum Heffter arrays is explained in the following result
which is an easy consequence of Theorem
\ref{thm:biembedding}.

\begin{cor}\label{preprecedente}
Let $A$ be a $^{\lambda}\N\mathrm{H}_t(m,n;h,k)$ such that $P(A)$ admits a solution $(\R,\C)$.
Then there exists a cellular biembedding of $^\lambda K_{(\frac{2nk}{\lambda t}+1)\times t}$ into an orientable surface whose face lengths are multiples of $k$ and $h$ strictly larger than $k$ and $h$ respectively.
\end{cor}
To present an existence result about biembeddings, we need to introduce some notation.

Given an $n \times n$ p.f. array $A$, for $i\in \{1,\dots,n\}$ we define the $i$-th diagonal of $A$ as follows
$$D_i=\{(i,1),(i+1,2),\ldots,(i-1,n)\}.$$
Here all the arithmetic on the row and column indices is performed modulo $n$, where $\{1,2,\ldots,n\}$ is the set of reduced residues.
The diagonals $D_{i+1}, D_{i+2}, \ldots, D_{i+k}$ are called $k$ \emph{consecutive diagonals}.
\begin{defi}
Let $n,k$ be integers such that $n\geq k\geq 1$. An $n\times n$ p.f. array $A$ is said to be:
\begin{itemize}
\item[1)] \emph{$k$-diagonal} if the non-empty cells of $A$ are exactly those of $k$ diagonals,
\item[2)] \emph{cyclically $k$-diagonal} if the non-empty cells of $A$ are exactly those of $k$ consecutive diagonals.
\end{itemize}
\end{defi}
We recall the following results about solutions of $P(A)$.
\begin{prop}[\cite{CDP}]\label{ElencoResults}
Given a cyclically $k$-diagonal $^{\lambda}\N\mathrm{H}_t(n;k)$, $A$, there exists a solution $(\R,\C)$ of $P(A)$ in the following cases:
\begin{itemize}
\item[1)] if $nk$ is odd, $n\geq k\geq 3$, and $gcd(n,k-1)=1$,
\item[2)] if $nk$ is odd, $n\geq k$, and $3<k<200$,
\item[3)] if $nk$ is odd, $n\geq k\geq 3$, and $n\geq (k-2)(k-1)$.
\end{itemize}
In all these cases $A$ also admits a pair of compatible orderings.
\end{prop}
We note that all the $^{\lambda}\N\mathrm{H}_t(n;k)$ obtained in Theorem \ref{thmonv} are cyclically $k$-diagonal. Therefore, it is enough to consider the $^{\lambda}\N\mathrm{H}_t(n;k)$'s over cyclic groups $\mathbb{Z}_v$, to obtain that:
\begin{thm}
Let $v$ be an even integer larger than or equal to $42$ or an odd one larger than or equal to $29$, $t$ be a divisor of $v$ and let $n\geq k\geq 3$ be odd integers such that $v=\frac{2nk}{\lambda}+t$ and that satisfy one of the hypothesis of Proposition \ref{ElencoResults}.
Then there exists a cellular biembedding of $^\lambda K_{(\frac{2nk}{\lambda t}+1)\times t}$ into an orientable surface whose face lengths are multiples of $k$ strictly larger than $k$.
\end{thm}
\section*{Acknowledgements}
The authors would like to thank Anita Pasotti and Tommaso Traetta for our useful discussions on this topic.  The first author was partially supported by INdAM--GNSAGA.

\end{document}